\documentclass[11pt]{amsart}
\usepackage[usenames,dvipsnames]{pstricks}
\usepackage[mathscr]{eucal}
\usepackage{amsfonts,amsmath,amssymb,amsthm,amscd,amsxtra}
\usepackage{enumerate,verbatim}
\usepackage[all,2cell,ps]{xy}
\usepackage[notcite]{}
\usepackage[pagebackref]{hyperref}
\usepackage{calc,color}
\usepackage{mathptmx}
\theoremstyle{plain} 

\newtheorem{thm}{Theorem}[section]
\newtheorem{prop}[thm]{Proposition}
\newtheorem{cor}[thm]{Corollary}
\newtheorem{lem}[thm]{Lemma}

\theoremstyle{definition}
\newtheorem{dfn}[thm]{Definition}

\newtheorem{conj}[thm]{Conjecture}

\newtheorem{rmk}[thm]{Remark}
\newtheorem{ex}[thm]{Example}

\numberwithin{equation}{section}

\newcommand{\fm}{\mathfrak{m}}

\newcommand{\fp}{\mathfrak{p}}

\newcommand{\fa}{\mathfrak{a}}

\DeclareMathOperator{\ann}{ann} 
 
\DeclareMathOperator{\V}{V} \DeclareMathOperator{\hh}{H}
\DeclareMathOperator{\e}{e} 
\DeclareMathOperator{\cm}{CM}

\def\x{\operatorname{\mathcal{X}}}
\def\A{\operatorname{\mathcal{A}}}
\def\y{\operatorname{\mathcal{Y}}}
\def\gd{\operatorname{\mathsf{G-dim}}}

\def\g{\operatorname{\mathcal{G}}}

\def\pd{\operatorname{\mathsf{pd}}}
\def\Pd{\operatorname{\mathsf{PD}}}
\def\gr{\operatorname{\mathsf{grade}}}
\def\Gr{\operatorname{\mathsf{gr}}}

\def\Tr{\mathsf{Tr}}

\def\D{\mathsf{D}}

\DeclareMathOperator{\coker}{Coker}
\DeclareMathOperator{\res}{\operatorname{\mathsf{res}}}
\DeclareMathOperator{\add}{\operatorname{\mathsf{add}}}
\DeclareMathOperator{\md}{\operatorname{\mathsf{mod}}}
\DeclareMathOperator{\CM}{\operatorname{\mathsf{CM}}}
\DeclareMathOperator{\NF}{\operatorname{\mathsf{NF}}}
\def\soc{\mathsf{soc}}
\def\rad{\mathsf{radius}}

\def\depth{\operatorname{\mathsf{depth}}}
 
\def\Ext{\operatorname{\mathsf{Ext}}}

 \DeclareMathOperator{\Spec}{Spec}

\def\Tor{\operatorname{\mathsf{Tor}}}

\DeclareMathOperator{\cod}{codim}

\def\XX{\mathrm{X}}

\def\urltilda{\kern -.15em\lower .7ex\hbox{\~{}}\kern .04em}
\def\urldot{\kern -.10em.\kern -.10em}\def\urlhttp{http\kern -.10em\lower -.1ex
\hbox{:}\kern -.12em\lower 0ex\hbox{/}\kern -.18em\lower
0ex\hbox{/}}

\numberwithin{equation}{thm}

\begin{document}
\baselineskip=15pt

\title[Resolving subcategories closed under certain operations]{Resolving subcategories closed under certain operations\\
and a conjecture of Dao and Takahashi}

\bibliographystyle{amsplain}
\author[A. Sadeghi]{Arash Sadeghi}
\address{Arash Sadeghi\\
	School of Mathematics, Institute for Research in Fundamental Sciences (IPM), P.O. Box: 19395-5746, Tehran, Iran}
\email{sadeghiarash61@gmail.com}

\author[R. Takahashi]{Ryo Takahashi}
\address{Ryo Takahashi\\
Graduate School of Mathematics, Nagoya University, Furocho, Chikusaku, Nagoya 464-8602, Japan}
\email{takahashi@math.nagoya-u.ac.jp}
\urladdr{http://www.math.nagoya-u.ac.jp/~takahashi/}
\thanks{2010 {\em Mathematics Subject Classification.} 13C60, 13D02}
\thanks{{\em Key words and phrases.} Cohen--Macaulay module, cosyzygy, radius, resolving subcategory, totally reflexive module}
\thanks{Sadeghi's research was supported by a grant from IPM. Takahashi was partly supported by JSPS Grants-in-Aid for Scientific Research 16H03923 and 16K05098.}

\begin{abstract}
Let $R$ be a commutative Noetherian local ring with residue field $k$. Let $\x$ be a resolving subcategory of finitely generated $R$-modules. This paper mainly studies when $\x$ contains $k$ or consists of totally reflexive modules. It is proved that $\x$ does so if $\x$ is closed under cosyzygies. A conjecture of Dao and Takahashi is also shown to hold in several cases.
\end{abstract}
\maketitle

\section{Introduction}
Throughout the paper, $R$ is a commutative Noetherian ring and $\md R$ denotes the category
of all finitely generated $R$-modules.
Whenever $R$ is assumed to be local, its unique maximal ideal is denoted by $\fm$.
All subcategories are full.
For a Cohen--Macaulay local ring $R$, we denote by $\cm(R)$ the
subcategory of $\md R$ consisting of maximal Cohen--Macaulay modules.
The radius of a subcategory was defined by Dao and Takahashi \cite{DT}.
This notion is linked to both the
representation theory and the singularity of $R$. For example, in \cite[Proposition 2.5]{DT}, over a Gorenstein complete local ring $R$, it is shown that the category of maximal Cohen--Macaulay modules has radius zero if and only if $R$ has finite Cohen--Macaulay representation type, in other words, $R$ is a simple hypersurface singularity (when the residue field of $R$ is algebraically closed field of characteristic zero).
In \cite[Theorem II]{DT}, for a Cohen--Macaulay complete local ring $R$ with perfect coefficient field, it is shown that $\cm(R)$ has finite radius. Dao and Takahashi
propose the following conjecture, and prove that it holds if R is a complete intersection.
\begin{conj}[Dao--Takahashi]\label{dtconj}
	Let $R$ be a Cohen--Macaulay local ring. Let $\x$ be a resolving subcategory of $\md R$ with finite radius. Then $\x$ is contained in $\cm(R)$.
\end{conj}
In the first part of this paper, we study the resolving subcategory with finite radius. We obtain the following result which supports Conjecture \ref{dtconj} in several cases (see Theorem \ref{t3}, Proposition \ref{p3} and Corollary \ref{c1}).

\begin{thm}\label{t1}
Let $R$ be a Cohen--Macaulay local ring, and let $\x$ be a resolving subcategory of $\md R$ with finite radius. Then $\x$ is contained in $\cm(R)$ (hence Conjecture \ref{dtconj} holds true) if one of the following holds.
\begin{enumerate}[(i)]
\item{$\x$ is closed under $R$-duals (i.e. $\x^*\subseteq\x$).}
\item{$R$ has minimal multiplicity.}
\item{$R$ is Gorenstein with $\cod R\leq 4$ or $\e(R)\leq11$ and each module in $\x$ has bounded Betti numbers.}
\end{enumerate}
\end{thm}
\noindent
Here $\cod R$ denotes the codimension of $R$ and $\e(R)$ stands for the multiplicity of $R$ with respect to the maximal ideal. A typical example of resolving subcategories closed under $R$-duals is $\g(R)$, the subcategory of totally reflexive modules.

It is quite natural to ask when a given resolving subcategory contains (some syzygy of) the residue field.
This is closely related to the vanishing problems of Tor and Ext modules.
Let $M$ be an $R$-module, and suppose that its resolving closure contains some syzygy of $k$.
Then $M$ is a so-called {\em test} module; an $R$-module $N$ has finite projective dimension whenever $\Tor_{\gg0}^R(M,N)=0$.
This question is also related to the radius of a resolving subcategory. For a resolving subcategory $\x$, over a local ring of positive dimension, if $\x$ contains the residue field, then it has infinite radius \cite[Theorem 4.4]{DT} (see also Proposition \ref{t2} and Lemma \ref{l2}).
In the second part of this paper, we study resolving subcategories which are closed under cosyzygies. A typical example of such subcategories is again $\g(R)$.
It is observed that, over a Cohen--Macaulay local ring $R$ of positive dimension, $\cm(R)$ is closed under cosyzygies if and only if $R$ is Gorenstein (i.e. $\cm(R)=\g(R)$).
We obtain the following result, as a special case of Theorem \ref{t5}.
\begin{thm}\label{tt}
Let $R$ be a local ring with residue field $k$, and let $\x$ be a resolving subcategory of $\md R$ closed under cosyzygies.
Then $\x$ either contains $k$ or is contained in $\g(R)$. In particular, if $\dim R>0$ and $\x$ has finite radius, then $\x\subseteq\g(R)$ and Conjecture \ref{dtconj} holds true.
\end{thm}
\noindent
We also prove that every resolving subcategory $\x$ with $\x=\Omega\x$ is contained in $\g(R)$ (see Theorem \ref{t4}). Here,  $\Omega$ stands for the syzygy functor. As a consequence of Theorem \ref{tt}, for a non-trivial resolving subcategory $\x$ of finite radius, over a Golod local ring $R$ of positive dimension, it is shown that $\x$ is closed under cosyzygies if and only if $R$ is a hypersurface.

The organization of this paper is as follows.
In section 2, we collect preliminary notions,
definitions and some known results which will be used in this paper. In section 3, we study the resolving subcategory with finite radius. We prove our first main result, Theorem \ref{t1}, in this section. In section 4, we study resolving subcategories which are closed under cosyzygies. We prove Theorem \ref{tt} in this section.
\section{Notations and Preliminary results}
Throughout this paper, $R$ denotes a commutative Noetherian ring and all $R$-modules are assumed to be finitely generated.
All subcategories are full and strict. (Recall that a subcategory $\x$ of a category $\mathcal{C}$ is called strict provided that for objects $M$ , $N\in\mathcal{C}$ with $M\cong N$, if  $M$ is in $\x$, then so is $N$.)
For a subcategory $\x$ of $\md R$, we denote by $\add\x$ the
additive closure of $\x$, namely, the subcategory of $\md R$ consisting of direct summands of finite direct sums of modules in $\x$. When $\x$ consists of a single module $M$, we simply denote it by $\add M$. Over a local ring $R$ with maximal ideal $\fm$, we denote by $\md_0(R)$ the subcategory of $\md(R)$ consisting of modules that are locally free on the punctured spectrum
$\Spec R\setminus \{\fm\}$. We say that two modules $X$ and $Y$ are \emph{stably isomorphic} and write $X\approx Y$ if $X\oplus P\cong Y\oplus Q$ for some projective modules $P$ and $Q$.

We recall the definition of an approximation with respect to a subcategory.
\begin{dfn}
	Let $\A$ be an additive category, and let $\x$ be a subcategory of $\A$.
	\begin{enumerate}[(i)]
		\item
		A morphism $f:X\to M$ (resp. $f:M\to X$) in $\A$ with $X\in\x$ is called a {\em right} (resp. {\em left}) {\em $\x$-approximation} of $M$ if for every morphism $f':X'\to M$ (resp. $f':M\to X'$) with $X'\in\x$ factors through $f$, that is, $f'=fg$ (resp. $f'=gf$) for some morphism $g:X'\to X$ (resp. $g:X\to X'$).
		\item
		A right (resp. left) $\x$-approximation $f:X\to M$ (resp. $f:M\to X$) is called {\em minimal} if every endomorphism $g:X\to X$ with $f=fg$ (resp. $f=gf$) is an automorphism.
	\end{enumerate}
\end{dfn}

\begin{rmk}
	For $\A=\md(R)$ and $\x=\add R$ one has the following.
	\begin{enumerate}[(i)]
		\item
		A right $\x$-approximation (resp. a minimal right $\x$-approximation) is nothing but a surjective homomorphism from a projective $R$-module (resp. a projective cover).
		\item
		Let $M$ be an $R$-module.
		Denote by $\lambda:M\to M^{**}$ the natural homomorphism.
		Take a surjection $\pi:P\to M^*$ with $P$ projective.
		Then the composition $\pi^*\circ\lambda:M\to P^*$ is a left $\x$-approximation.
	\end{enumerate}
\end{rmk}

In the following we recall the definitions of the syzygy, transpose and cosyzygy of a given module, which will be used throughout the paper.
\begin{dfn}
	Let $M$ be an $R$-module.	
	\begin{enumerate}[(i)]
		\item{Assume that $P_0\overset{d_0}{\rightarrow}M$ is a right $\add(R)$-approximation of $M$. Then, the kernel of $d_0$ is called the {\em first syzygy} of $M$ and denoted by
			$\Omega^1M$ (or $\Omega_R^1M$) which is unique up to projective equivalence. Inductively, we define the {\em $n$-th syzygy} module of $M$, $\Omega^{n}M:=\Omega^{1}(\Omega^{n-1}M)$ for all $n>0$. For convention, we set 	$\Omega^0M = M$.}
		\item{Let $P_1\rightarrow P_0\overset{d_0}{\rightarrow}M\rightarrow0$ be a projective presentation of $M$. The cokernel of the $R$-dual map $d_0^*:P_0^*\rightarrow P_1^*$ is called the {\em (Auslander) transpose} of $M$ and denoted by $\Tr M$ (or $\Tr_RM$). Therefore we have the following exact sequence:
			$$0\rightarrow M^*\rightarrow P_0^*\rightarrow P_1^*\rightarrow\Tr M\rightarrow0.$$
		The transpose of a module is unique up to projective equivalence.}
		\item{Let $M\overset{d_{-1}}\rightarrow P_{-1}$
			be a left $\add(R)$-approximation of $M$. Then we call the cokernel of $d_{-1}$ the {\em first cosyzygy} of $M$ and denote it by $\Omega^{-1}M$ (or $\Omega^{-1}_RM$). Inductively, we define the {\em $n$-th cosyzygy} module of $M$, $\Omega^{-n}M:=\Omega^{-1}(\Omega^{-(n-1)}M)$ for all $n>0$. }
	\end{enumerate}	
Whenever $R$ is local, we define the syzygy, cosyzygy and transpose by using a minimal right $\add(R)$-approximation, a minimal left $\add(R)$-approximation and a minimal projective presentation, respectively so that they are uniquely determined up to isomorphism.
\end{dfn}

\begin{lem}\label{le}
One has $\Omega^{-1}M\approx\Tr\Omega\Tr M$ for all $R$-modules $M$.
\end{lem}
\begin{proof}
Let $P_1\to P_0\xrightarrow{\pi} M\to0$ be a projective presentation of an $R$-module $M$.
Dualizing this by $R$ gives an exact sequence $0\to M^*\xrightarrow{\pi^*}P_0^*\to\Omega\Tr M\to0$.
Taking a surjection $\varepsilon:Q\to M^*$ with $Q$ projective, we get a projective presentation $Q\xrightarrow{\pi^*\varepsilon}P_0^*\to\Omega\Tr M\to0$, which induces an exact sequence $P_0^{**}\xrightarrow{\varepsilon^*\pi^{**}}Q^*\to\Tr\Omega\Tr M\to0$.
For each module $X$, denote by $\lambda_X:X\to X^{**}$ the natural homomorphism.
As $\lambda_{P_0}:P_0\to P_0^{**}$ is an isomorphism, we have an exact sequence $P_0\xrightarrow{\varepsilon^*\pi^{**}\lambda_{P_0}}Q^*\to\Tr\Omega\Tr M\to0$.
Thus we get a commutative diagram
$$
\xymatrix{
	M\ar[r]^{\varepsilon^*\lambda_M} & Q^*\ar[r] & \Omega^{-1}M\ar[r] & 0 \\
	P_0\ar@{->>}[u]^{\pi}\ar[r]^{\varepsilon^*\pi^{**}\lambda_{P_0}} & Q^*\ar@{=}[u]\ar[r] & \Tr\Omega\Tr M\ar[r] & 0
	}$$
with exact rows.
This induces an isomorphism $\Tr\Omega\Tr M\overset{\cong}{\to}\Omega^{-1}M$.
\end{proof}
In the following, we collect some basic properties of syzygies and transposes which will be used throughout the paper basically without reference (see \cite{DT1} for more details).
\begin{prop}\label{p1}
	For an $R$-module $M$ the following hold:
	\begin{enumerate}[(i)]
		\item{Let $0\rightarrow L\rightarrow M\rightarrow N\rightarrow0$ be an exact sequence of $R$-modules. Then one has exact sequences
			$$0\rightarrow\Omega N\rightarrow L\rightarrow M\rightarrow0,$$
			$$0\rightarrow\Omega L\rightarrow\Omega M\rightarrow \Omega N\rightarrow0,$$
			$$0\rightarrow N^*\rightarrow M^*\rightarrow L^*\rightarrow\Tr N\rightarrow\Tr M\rightarrow\Tr L\rightarrow0$$
		of $R$-modules (up to projective summands).}
		\item{For an integer $n>0$, there exists an exact sequence
			$$0\rightarrow\Ext^n_R(M,R)\rightarrow\Tr\Omega^{n-1}M\rightarrow\Omega\Tr\Omega^nM\rightarrow0.$$
		}
		\item{There are stable isomorphisms $(\Tr M)^*\approx\Omega^2M$ and $M^*\approx\Omega^2\Tr M$.}
	\end{enumerate}
\end{prop}

The notion of Gorenstein dimension was introduced by Auslander \cite{A1}, and deeply developed by Auslander and Bridger in \cite{AB}.
\begin{dfn}
An $R$-module $M$ is called of Gorenstein dimension at most zero (or totally reflexive) if $M\cong M^{\ast\ast}$ (or equivalently, the canonical map $M\rightarrow M^{**}$ is bijective) and $\Ext^i_R(M,R)=\Ext^i_R(M^*,R)=0$ for all $i>0$.
We denote by $\g(R)$ the subcategory of $\md R$ consisting of totally reflexive modules.
\end{dfn}
The \emph{Gorenstein dimension} of $M$, denoted by $\gd_R(M)$, is defined to be the infimum of all
nonnegative integers $n$ such that there exists an exact sequence
$$0\rightarrow G_n\rightarrow\cdots\rightarrow G_0\rightarrow  M \rightarrow 0,$$
where each $G_i$ are totally reflexive. Note by definition that $\gd_R(0)=-\infty$.

In the following, we summarize some basic facts about Gorenstein dimension (see \cite{AB}).
\begin{prop}\label{G}
	For an $R$-module $M$, the following statements hold true.
	\begin{enumerate}[(i)]
		\item{$\gd_R(M)\leq0$ if and only if $\gd_R(\Tr M)\leq0$;}
		\item{(Auslander-Bridger formula) If $M$ has finite Gorenstein dimension, then $$\gd_R(M)=\depth R-\depth_R(M).$$}
		\item{If $\gd_R(M)$ is finite, then $\gd_R(M)=\sup\{i\mid\Ext^i_R(M,R)\neq0\}.$}
		\item{$R$ is Gorenstein if and only if $\gd_R(M)<\infty$ for all finite $R$-module $M$.}
	\end{enumerate}
\end{prop}
Let $n$ be a positive integer. Recall that an $R$-module $M$ is called {\em $n$-torsion free} if
$\Ext^i_R(\Tr M,R)=0$ for all $1\leq i\leq n$.
An $R$-module $M$ is said to be an {\em$n$-th syzygy} if there exists an exact sequence
$$0\to M \to P_0\to\cdots\to P_{n-1}$$
with the $P_0,\cdots, P_{n-1}$ are projective. By convention, every module is a $0$-th syzygy. We denote by $\Omega^n(\md(R))$
the full subcategory of $\md(R)$ consisting of
$n$-th syzygy modules. For an integer $m$, we define a set $\XX^m(R)=\{\fp\in\Spec R\mid \depth R_\fp\leq m\}$.
\begin{lem}\cite[Theorem 43]{M1}\label{t7}
Let $R$ be a ring and let $M$ be an $R$-module. Assume that $n$ is an integer such that $\gd_{R_\fp}(M_\fp)<\infty$ for all $\fp\in \XX^{n-2}(R)$.
Then $M$ is $n$-torsion free if and only if $M$ is an $n$-th syzygy module.
\end{lem}

The notion of a resolving subcategory has been introduced by Auslander and Bridger \cite{AB}.
\begin{dfn}
A subcategory $\x$ of $\md R$ is called \emph{resolving} if the following hold:
\begin{enumerate}[(i)]
	\item{$\x$ contains the projective $R$-modules.}
	\item{$\x$ is closed under direct summands (i.e. if $M$ is an $R$-module in $\x$ and $N$ is an $R$-module that is a direct summand of $M$, then $N$ is also in $\x$).}
   \item{$\x$ is closed under extensions (i.e. for an exact sequence $0\rightarrow L\rightarrow M\rightarrow N\rightarrow 0$ of $R$-modules, if $L$, $N$ are in $\x$, then so is $M$).}
    \item{$\x$ is closed under kernels of epimorphisms (i.e. for an exact sequence $0\rightarrow L\rightarrow M\rightarrow N\rightarrow 0$ of $R$-modules, if $M$, $N$ are in $\x$, then so is $L$).}
\end{enumerate}	
\end{dfn}
A subcategory $\x$ of $\md R$ is resolving if and only if $\x$ contains $R$ and is
closed under direct summands, extensions and syzygies \cite[Lemma 3.2]{Y}. There are many examples of resolving subcategories in $\md(R)$.
For example, $\cm(R)$ is a resolving subcategory of $\md R$ if $R$ is Cohen--Macaulay. The subcategory of $\md R$ consisting of totally reflexive $R$-modules is resolving by \cite[3.11]{AB} (see also \cite[Example 1.6]{T} for more examples).

A subcategory $\x$ of $\g(R)$ is called \emph{thick} if
$\x$ is closed under direct summands and that for each exact sequence $0\rightarrow L\rightarrow M\rightarrow N\rightarrow 0$ of totally reflexive $R$-modules, if two of $L$, $M$, $N$ are in $\x$, then so is the
third. Note that a thick subcategory of $\g(R)$
containing $R$ is a resolving subcategory of $\md R$.

Next we recall the definition of the radius of a subcategory defined by Dao and Takahashi \cite{DT}.
\begin{dfn}
Let $\x$ and $\y$ be subcategories of $\md(R)$.	
\begin{enumerate}[(i)]
	\item{The subcategory $[\x]$ is defined as follows:
		  $$[\x]=\add\{R,\Omega^iX\mid X\in\x,\,i\ge0\}.$$
		 When $\x$ consists of a single module $X$, we simply denote it by $[X]$.}
	\item{The subcategory of $\md R$
		consisting of the $R$-modules $M$ which fits into an exact sequence $0\rightarrow X\rightarrow M\rightarrow Y\rightarrow0$ with $X\in\x$ and $Y\in\y$ is denoted by $\x\circ\y$. We set
		$\x\bullet\y=\left[[\x]\circ[\y]\right].$}
	\item{The \emph{ball of radius $n$ centered at} $\x$ is defined as follows:
\begin{equation*}
[\x]_n =
\begin{cases}
[\x] &\text{for $n=1$}\,,\\
[\x]_{n-1}\bullet[\x]=[[\x]_{n-1}\circ[\x]]  &\text{for $n>1$}\,.\\
\end{cases}
\end{equation*}
If $\x$ consists of a single module $X$, then we simply denote $[\x]_n$ by $[X]_n$, and call it the ball of radius $n$ centered at $X$. We write $[\x]^R_n$ when we should specify that $\md R$ is
the ground category where the ball is defined. Note that, by \cite[Proposition 2.2]{DT}, $$[\x]_n=\overset{n}{\overbrace{\x\bullet\cdots\bullet\x}}.$$}
\item{The \emph{radius} of $\x$, denoted by $\rad\x$, is the infimum of the integers $n\geq0$ such that
	there exists a ball of radius $n+1$ centered at a module belonging to $\x$.}
\end{enumerate}	
\end{dfn}
The following result will be used throughout the paper.
\begin{lem}\cite[Propsition 4.9]{DT}\label{t}
Let $R$ be a local ring and let $\x$ be a resolving subcategory of $\md(R)$. If $\x$ contains a module $M$ with $0<\pd_RM<\infty$, then $\rad(\x)=\infty$.
\end{lem}	
\section{Resolving subcategories with finite radius}
In this section, we study resolving subcategories with finite radius. We prove our first main result, Theorem \ref{t1}.
For an ideal $I$ of $R$ and an $R$-module $M$, we denote by
$\fa^i_I(M)$ the annihilator $\ann_R(\hh^i_{I}(M))$ of the $i$-th local cohomology module of $M$.
The following result is a generalization of \cite[Lemma 4.2]{DT}.
\begin{lem}\label{l1}
Let $R$ be a ring and let $I$ be an ideal of $R$. Assume that $M$ and $C$ are $R$-modules and that $n$ is a positive integer.
If $M\in[C]_n$, then $$\left(\underset{i=0}{\overset{t}{\prod}}\fa_I^i(R)\,\fa_I^i(C)\right)^n\subseteq\fa_I^t(M) \text{ for all } t\geq0.$$
\end{lem}
\begin{proof}
We argue by induction on $n$. If $n=1$, then $M$ is isomorphic to a direct summand of a finite direct sum of copies of $R\oplus(\underset{j=0}{\overset{m}{\bigoplus}}\Omega^jC)$.
Therefore $\hh^t_I(M)$ is isomorphic to a direct summand of a finite direct sum of copies of $\hh^t_I(R)\oplus(\underset{j=0}{\overset{m}{\bigoplus}}\hh^t_I(\Omega^jC))$. Hence
\begin{equation}\tag{\ref{l1}.1}
\fa_I^t(R)\cap\left(\underset{j=0}{\overset{m}{\bigcap}}\fa_I^t(\Omega^jC)\right)\subseteq\fa_I^t(M).
\end{equation}
For each $j>0$ there is an  exact sequence $0\rightarrow\Omega^{j}C\rightarrow P_{j-1}\rightarrow\Omega^{j-1}C\rightarrow0$, where $P_{j-1}$ is a projective $R$-module. This induces an exact sequence
$\hh^{i-1}_I(\Omega^{j-1}C)\rightarrow\hh^i_I(\Omega^{j}C)\rightarrow\hh^{i}_I(P_{j-1})$, which gives the following inclusions
$$\fa_I^t(R)\fa_I^{t-1}(R)\cdots\fa_I^{t-j+1}(R)\fa_I^{t-j}(C)\subseteq\cdots\subseteq\fa_I^t(R)\fa_I^{t-1}(R)\fa_I^{t-2}(\Omega^{j-2}C)\subseteq\fa_I^t(R)\fa_I^{t-1}(\Omega^{j-1}C)\subseteq\fa_I^t(\Omega^{j}C)$$
for all $j\leq t$. Also
$$\fa_I^t(R)\fa_I^{t-1}(R)\cdots\fa_I^{1}(R)\fa_I^{0}(R)\subseteq\fa_I^t(R)\fa_I^{t-1}(R)\cdots\fa_I^{1}(R)\fa_I^{0}(\Omega^{j-t}C)\subseteq\cdots\subseteq\fa_I^t(R)\fa_I^{t-1}(\Omega^{j-1}C)\subseteq\fa_I^t(\Omega^{j}C)$$ for all $j>t$.
Therefore the assertion is clear for the case $n=1$ by (\ref{l1}.1). Now let $n>1$. We have $M\in[C]_n=[C]_{n-1}\bullet[C]$, and by \cite[Proposition 2.2(1)]{DT}, there exists an exact sequence $0\rightarrow X\rightarrow Y\rightarrow Z\rightarrow0$ with $X\in[C]_{n-1}$ and $Z\in[C]$ such that $M$ is a direct summand of $Y$.
The induction hypothesis shows
$$\left(\underset{i=0}{\overset{t}{\prod}}\fa_I^i(R)\,\fa_I^i(C)\right)^{n}=\left(\underset{i=0}{\overset{t}{\prod}}\fa_I^i(R)\,\fa_I^i(C)\right)^{n-1}\left(\underset{i=0}{\overset{t}{\prod}}\fa_I^i(R)\,\fa_I^i(C)\right)\subseteq\fa^t_I(X)\,\fa^t_I(Z)\subseteq\fa^t_I(Y)\subseteq\fa^t_I(M),$$
which completes the proof.
\end{proof}
The following result can be viewed as a generalization of \cite[Theorem 4.4]{DT}.
\begin{prop}\label{t2}
Let $(R,\fm,k)$ be a local ring and let $\x$ be a resolving subcategory of $\md(R)$. Assume that $\rad(\x)<\infty$ and that $\Omega^tk\in\x$ for some $t\geq0$. Then $\dim R\leq t$.	
\end{prop}
\begin{proof}
Put $n=\rad(\x)$. By definition, there exists an $R$-module $C$ such that $\x\subseteq[C]_{n+1}$. In particular, $\Omega^tk\in[C]_{n+1}$. We claim that $\Omega^tL\in\x$ for all finite length $R$-module $L$. We prove the claim by induction on $m=\ell(L)$. If $m=1$, then $L\cong k$ and so the assertion is clear. Now let $m>1$ and consider the following exact sequence $0\rightarrow L_1\rightarrow L\rightarrow k\rightarrow0$, where $\ell(L_1)=m-1$. From the above exact sequence we get the following exact sequence $0\rightarrow\Omega^t L_1\rightarrow\Omega^t L\oplus F\rightarrow\Omega^tk\rightarrow0$,
where $F$ is a free $R$-module. It follows from the above exact sequence and the induction hypothesis that $\Omega^t L\in\x$ and the claim is proved.
In particular, $\Omega^t_R(R/\fm^i)\in[C]_{n+1}^R$ for all $i>0$. Therefore $\Omega^t_{\hat{R}}(\hat{R}/\hat{\fm}^i)\cong\hat{\Omega^t_{R}(R/\fm^i)}\in[C]_{n+1}^{\hat{R}}$. Set $I=\underset{i=0}{\overset{t}{\prod}}\left(\fa^i_{\hat{\fm}}(\hat{R}).\fa^i_{\hat{\fm}}(\hat{C})\right)^{n+1}$. It follows from Lemma \ref{l1} that $I\subseteq\fa^t_{\hat{\fm}}(\Omega^t_{\hat{R}}(\hat{R}/\hat{\fm}^i))$. Now it is easy to see that
$I^{t+1}\subseteq\fa^0_{\hat{\fm}}(\hat{R}/\hat{\fm}^i)=\hat{\fm}^i$ for all $i>0$. Hence $I^{t+1}=0$ and so
$$\dim R=\dim \hat{R}=\dim (\hat{R}/I^{t+1})=\dim \hat{R}/I\leq t$$ by \cite[Theorem 8.1.1]{BH}.
\end{proof}
Let $\x$ be a subcategory of $\md(R)$. We denote by $\res\x$ (or $\res_R \x$) the resolving closure
of $\x$, namely, the smallest resolving subcategory of $\md(R)$ containing $\x$. If $\x$ consists of a single module $M$, then we simply denote it by $\res(M)$ (or $\res_R(M)$).
\begin{lem}\label{l2}
Let $(R,\fm,k)$ be a Cohen--Macaulay local ring of dimension $d$ and let $\x$ be a resolving subcategory of $\md(R)$ with $\rad(\x)<\infty$. If $\Omega^nk\in\x$ for some $n\geq0$, then $\x\subseteq\cm(R)$.	
\end{lem}	
\begin{proof}
Assume contrarily that $\x\nsubseteq\cm(R)$. Hence there exists an $R$-module $M\in\x$ such that $t=\depth_R(M)<d$.
By \cite[Proposition 4.2]{DT1}, $\Omega^tk\in\res(M\oplus\Omega^nk)\subseteq\x$.
It follows from this and Proposition \ref{t2} that $d\leq t$ which is a contradiction.
\end{proof}	
The following lemma plays a crucial role in the proof of our main results.
\begin{lem}\label{l6}
	Let $(R,\fm,k)$ be a local ring and let $\x$ be a resolving subcategory of $\md(R)$. Assume that $\Ext^i_R(\x,R)\neq0$ for some $i>0$. Then there exists an $R$-module $M\in\x\cap\md_0(R)$ such that $\Ext^i_R(M,R)\neq0$.
\end{lem}
\begin{proof}
	Assume that	$\Ext^i_R(M,R)\neq0$ for some $M\in\x$.
	If $\NF(M)=\V(\fm)$, then we have nothing to prove. So assume that $\fp\in\NF(M)$ for some $\fp\in\Spec R\setminus\{\fm\}$ and that $x$ is an element in $\fm$ which is not in $\fp$. By \cite[Proposition 4.2]{T1}, there is a commutative diagram:
	$$\begin{CD}
	0@>>> \Omega M@>>> R^n @>>> M @>>>0  \\
	&& @VV{x}V@VVV@VV{\parallel}V \\
	0@>>> \Omega M@>>> M_1@>>> M@>>>0\\
	\end{CD}$$
	where $M_1$ is an $R$-module such that $\V(\fm)\subseteq\NF(M_1)\subseteq\NF(M)$ and $\D(x)\cap\NF(M_1)=\emptyset$. Hence $\V(\fm)\subseteq\NF(M_1)\subsetneq\NF(M)$.
	It follows from the above diagram that $M_1\in\res(M)\subseteq\x$. The above diagram induces the following commutative diagram:
	$$\begin{CD}
	\Ext^{i-1}_R(\Omega M,R)@>>>\Ext^i_R(M,R) @>>>\Ext^i_R(M_1,R)  \\
	@VV{x}V@VV{\parallel}V@VVV \\
	\Ext^{i-1}_R(\Omega M,R)@>>>\Ext^i_R(M,R)@>>>0\\
	\end{CD}$$
	If $\Ext^i_R(M_1,R)=0$, then it follows from the above diagram and the Nakayama's Lemma that $\Ext^i_R(M,R)=0$ which is a contradiction. Therefore $\Ext^i_R(M_1,R)$ is a nonzero $R$-module.
	If $\V(\fm)$ coincides with $\NF(M_1)$, then we are done. So we assume that $\V(\fm)$ is strictly contained in $\NF(M_1)$. Then,
	a similar argument to the above shows that there exists an $R$-module $M_2\in\x$ which satisfies $$\V(\fm)\subseteq\NF(M_2)\subsetneq\NF(M_1)\subsetneq\NF(M),$$
	and $\Ext^i_R(M_2,R)\neq0$. By \cite[Corollary 2.11(1)]{T1}, all nonfree loci are closed subsets of $\Spec R$.
	Since $\Spec R$ is a Noetherian space, every descending chain of closed subsets stabilizes.
	This means that the above procedure to construct modules $M_i$ cannot be iterate infinitely many times. Hence there exists an $R$-module $N\in\x$ such that $\NF(N)=\V(\fm)$ and $\Ext^i_R(N,R)\neq0$.
\end{proof}	
\begin{thm}\label{t3}
Let $(R,\fm,k)$ be a local ring and let $\x$ be a resolving subcategory
of $\md(R)$. Assume that $\rad(\x)<\infty$	and that $\x^*\subseteq\x$. Then the following statements hold:
\begin{enumerate}[(i)]
\item{If $\dim R\geq 3$, then $\x\subseteq\g(R)$.}
\item{If $R$ is Cohen--Macaulay, then $\x\subseteq\cm(R)$.}
\end{enumerate}
\end{thm}
\begin{proof}
(i). First assume that $\Ext^1_R(\x,R)=0$ and that $M\in\x$. As $\x$ is resolving and closed under dual, $\Ext^i_R(M,R)=0=\Ext^i_R(M^*,R)$ for all $i>0$. The exact sequence $0\rightarrow\Omega M\rightarrow F\rightarrow M\rightarrow0$ induces the following commutative diagram with exact rows:
$$\begin{CD}
0@>>> \Omega M@>>> F_0 @>>> M @>>>0  \\
&& @VV{f}V@VV{\|\wr}V@VV{g}V \\
0@>>> (\Omega M)^{**}@>>> F_0^{**}@>>> M^{**}@>>>0\\
\end{CD}
$$
It follows from the above diagram and the snake Lemma that $f$ is a monomorphism and $g$ is an epimorphism. Similarly, the exact sequence $0\rightarrow\Omega^2 M\rightarrow F\rightarrow \Omega M\rightarrow0$ induces the following commutative diagram with exact rows:
$$\begin{CD}
0@>>> \Omega^2 M@>>> F_1 @>>> \Omega M @>>>0  \\
&& @VVV@VV{\|\wr}V@VV{f}V \\
0@>>> (\Omega^2 M)^{**}@>>> F_1^{**}@>>> (\Omega M)^{**}@>>>0\\
\end{CD}
$$
It follows from the above diagram and the snake Lemma that $f$ is  is an epimorphism. Hence $f$ is an isomorphism and so is $g$. In other words, $\x\subseteq\g(R)$. Now assume that $\Ext^1_R(\x,R)\neq0$. By Lemma \ref{l6} there exists an $R$-module $M\in\x\cap\md_0(R)$ with $\Ext^1_R(M,R)\neq0$.
Let $\sigma\in\soc(\Ext^1_R(M,R))$ be a non-zero element. Hence we have the following non-split sequence: $0\rightarrow R\rightarrow N\rightarrow M\rightarrow0$ which induces the following exact sequence:
\begin{equation}\tag{\ref{t3}.1}
0\rightarrow M^*\rightarrow N^*\rightarrow R\rightarrow k\rightarrow0.
\end{equation}
As $\x$ is closed under dual, $M^*$ and $N^*$ are belongs to $\x$. It follows from the exact sequence (\ref{t3}.1) and \cite[Lemma 4.3]{T} that
$\Omega^2k\in\x$ which is a contradiction by Proposition \ref{t2}.
	
(ii) Assume contrarily that $\x\nsubseteq\cm(R)$. Hence $\Omega^nk\notin\x$ for all $n\geq0$, by Lemma \ref{l2}. By Auslander-Bridger formula $\x\nsubseteq\g(R)$. Hence, by the similar argument in the proof of first part, one can show that $\Omega^2k\in\x$ which is a contradiction.
\end{proof}
\begin{prop}\label{p2}
Let $R$ be a local ring, and $\x$ a resolving subcategory of $\md R$.
Then $\x$ has infinite radius if $\x$ contains a module $M$ such that:\par
(i) $\Ext^1_R(M,R)\ne0$,\quad
(ii) $\Omega M$ is reflexive,\quad
(iii) $\Omega^{-2}\Omega M$ is in $\x$.
\end{prop}
\begin{proof}
By the similar argument to \cite[Theorem 4.10 (1)]{DT}, we have the following exact sequence:
\begin{equation}\tag{\ref{p2}.1}
0\rightarrow M\rightarrow X\rightarrow\Omega^{-2}\Omega M\rightarrow0,
\end{equation}
where $\pd_R(X)\leq1$.
It follows from the above exact sequence that $X\in\x$. Note that $\Ext^i_R(\Omega^{-2}\Omega M,R)=0$ for $i=1, 2$. Therefore, the exact sequence (\ref{p2}.1) induces the following isomorphism $\Ext^1_R(X,R)\cong\Ext^1_R(M,R)\neq0$. In other words, $\x$ contains a non-free $R$-module of finite projective dimension. Now the assertion follows from Lemma \ref{t}.
\end{proof}
\begin{rmk}
The above proposition recovers \cite[Theorem 3.3(2)]{DT}.
\end{rmk}
Let $I$ be an ideal of $R$ and $M$ an $R$-module.
We define the {\em Loewy length} of $M$ with respect to $I$ by $\ell\ell_I(M)=\inf\{n\ge0\mid I^nM=0\}$.
The following result is a generalization of \cite[Proposition 4.3]{DT}.
\begin{lem}\label{l7}
Let $I$ be an ideal of $R$.
Let $\x$ be a resolving subcategory of $\md R$ having finite radius.
Then $\sup_{M\in\x}\{\ell\ell_I(\Gamma_I(M))\}<\infty$.
\end{lem}
\begin{proof}
By assumption there are an $R$-module $C$ and an integer $n>0$ such that $\x\subseteq[C]_n$.
Choose integers $r,c>0$ such that $I^r,I^c$ annihilate $\Gamma_I(R),\Gamma_I(C)$ respectively.
Fix any module $M\in\x$.
Applying Lemma \ref{l1}, the ideal $\fa_I^0(M)$ contains $(\fa_I^0(R)\fa_I^0(C))^n$, which contains $(I^rI^c)^n=I^{(r+c)n}$.
Hence $\sup_{M\in\x}\{\ell\ell_I(\Gamma_I(M))\}\le(r+c)n$.
\end{proof}
\begin{prop}
Let $R$ be a local ring.
Let $\x$ be a resolving subcategory of $\md R$.
Suppose that there exist an $R$-module $M$ and an $M$-sequence $\underline x=x_1,\dots,x_n$ such that $M/\underline xM$ is in $\x$.
Then $\x$ has infinite radius.
\end{prop}
\begin{proof}
Let $I$ be the ideal of $R$ generated by $\underline x$.
Then $IM\ne M$, so $I$ is contained in $\fm$.
Since $\underline x$ is $M$-regular, there is a natural isomorphism $\Gr_I(M)\cong M/IM[X_1,\dots,X_n]$; see \cite[Theorem 1.1.8]{BH}.
This implies that for all $i>0$ the module $I^{i-1}M/I^iM$ is a direct sum of copies of $M/IM$.
Since $M/IM$ is in $\x$, so is $I^{i-1}M/I^iM$, and so is $M/I^iM$.
Suppose that $\x$ has finite radius.
Then Lemma \ref{l2} implies that $\ell:=\sup_{M\in\x}\{\ell\ell_I(\Gamma_I(M))\}$ is finite.
Hence $I^\ell$ annihilates $\Gamma_I(M/I^iM)=M/I^iM$, which means that $I^\ell M$ is contained in $I^iM$.
Thus $I^\ell M$ is contained in $\bigcap_{i>0}I^iM$, which is zero by Krull's intersection theorem.
In particular, we have $x_1^\ell M=0$.
This contradicts the fact that $x_1^\ell$ is a non-zerodivisor on $M$.
\end{proof}
Following \cite{NT}, we say that an ideal $I$ of $R$ is {\em quasi-decomposable} if $I$ contains an $R$-sequence $\underline x=x_1,\dots,x_n$ such that the module $I/\underline xR$ is decomposable.
A Cohen--Macaulay non-Gorenstein local ring with minimal multiplicity has quasi-decomposable maximal ideal.
We denote by $\Pd(R)$ the subcategory
of $\md(R)$ consisting of modules of finite projective dimension.
\begin{prop}\label{p3}
Let $(R,\fm,k)$	be a Cohen--Macaulay local ring with quasi-decomposable maximal ideal. Let $\x$ be a resolving subcategory of $\md(R)$ with $\rad(\x)<\infty$.
Then $\x\subseteq\cm(R)$.
\end{prop}	
\begin{proof}
Set $\mathcal{Y}=\x\cap\cm(R)$. Note that $\mathcal{Y}$ is a resolving subcategory contained in $\cm(R)$. If $\mathcal{Y}\neq\add(R)$, then by \cite[Lemma 4.4]{NT}, $\Omega^dk\in\x$ and so the assertion is clear by Lemma \ref{l2}.
Now assume that $\mathcal{Y}=\add(R)$. Hence $\x\subseteq\Pd(R)$.
As $\rad(\x)<\infty$, $\x=\add(R)\subseteq\cm(R)$ by Lemma \ref{t}.
\end{proof}
Let $R$ be a Cohen--Macaulay local ring of dimension $d$. Assume that $I$ is an $\fm$-primary ideal of $R$ and that
$\Gr_I(R)=\underset{n\geq0}{\oplus} I^n/I^{n+1}$ is the associated graded ring of $I$. Let $\fa(\Gr_I(R))$ denote the $\fa$-invariant of $\Gr_I(R)$ \cite[Definition 3.1.4]{GW}. Following \cite{GOTWY}, we say that $I$ is an \emph{Ulrich ideal} of $R$ if $\Gr_I(R)$ is a Cohen--Macaulay ring with $\fa(\Gr_I(R))\leq 1-d$ and $I/I^2$ is a free $R/I$-module. For each $R$-module $M\in\CM(R)$, we have $\mu(M)\leq\e(M)$, where $\mu(M)$ denotes the number of minimal generators of $M$ and $\e(M)$ is the multiplicity of $M$ with respect to the maximal ideal.
A maximal Cohen--Macaulay $R$-module $M$ is called \emph{Ulrich module} if $\mu(M)=\e(M)$ (i.e. $M$ has the maximum number of generators with respect to the above inequality).

\begin{prop}
Let $(R,\fm,k)$ be a Cohen--Macaulay local ring and let $\x$ be a resolving subcategory of $\md(R)$.
\begin{enumerate}[\rm(1)]
\item
If $\rad(\x)<\infty$ and $\x$ contains at least one Ulrich module, then $\x\subseteq\cm(R)$.
\item
If $R/I\in\x$ for some Ulrich ideal $I$, then $\rad(\x)=\infty$.
\end{enumerate}
\end{prop}	
\begin{proof}
(1) Assume that $M\in\x$ is an Ulrich module. Then there exists a parameter ideal $\underline{x}$ of $M$ such that $\fm M=\underline{x}M$. Therefore $M/\underline{x}M\cong\overset{n}{\oplus}k$ for some $n>0$. As $M$ is Cohen--Macaulay, $\underline{x}$ is an $M$-regular sequence.
Set $t=\depth_R(M)$ and $M_i=M/(x_1,\cdots,x_i)M$ for $0\leq i\leq t$. The exact sequence $0\rightarrow M_i\overset{x_{i+1}}{\rightarrow} M_i\rightarrow M_{i+1}\rightarrow0$, induces the following exact sequence
$0\rightarrow\Omega^{i}M_i\rightarrow \Omega^iM_i\oplus F_i\rightarrow\Omega^i M_{i+1}\rightarrow0$, where $F_i$ is a free $R$-module.
Hence we get the following exact sequence
$$0\rightarrow\Omega^{i+1}M_{i+1}\rightarrow \Omega^iM_i\oplus F'_i\rightarrow\Omega^i M_{i}\oplus F_i\rightarrow0,$$
where $F'_i$ is a free $R$-module. Now by induction on $t$, one can show that $\Omega^t(M/\underline{x}M)\in\x$. Therefore
$\Omega^tk\in\x$. Now the assertion follows from Lemma \ref{l2}.

(2) Assume contrarily that $\rad(\x)<\infty$.
If $\dim R=1$, then $I^2=aI$ for some regular element $a$ by \cite[Remark 3.1.6]{GW}. Therefore, $I\cong I^n$ for all $n>0$. Also
$$I^i/I^{i+1}\cong a^{i-1}I/a^{i-1}I^2\cong I/I^2\cong\oplus R/I\in\x.$$
It follows from the exact sequence $0\rightarrow I^i/I^{i+1}\rightarrow R/I^{i+1}\rightarrow R/I^{i}\rightarrow0$ that $R/I^{i}\in\x$ for all $i>0$. Hence, by Lemma \ref{l7}, there exists a positive integer $n$ such that $\fm^n\subseteq I^i$ for all $i>0$. In other words, $\fm^n=0$ which is a contradiction.
Assume that $d\geq2$. By \cite[Lemma 2.3]{GOTWY} and \cite[Remark 3.1.6]{GW}, there exists a parameter ideal $Q\subseteq I$ such that $I/Q$ is a free $R/I$-module. Hence $I/Q\in\x$. It follows from the exact sequence
$0\rightarrow Q\rightarrow I\rightarrow I/Q\rightarrow0$ that
$Q\in\x$. Note that $\pd_R(Q)=\pd_R(R/Q)-1=d-1>0$ which is a contradiction by Lemma \ref{t}.
\end{proof}
\begin{prop}\label{pr1}
Let $R$ be a local ring and $\x$ a resolving subcategory of $\md(R)$.
Assume that $\x$ contains an eventually periodic module of finite and positive G-dimension.
Then $\rad(\x)=\infty$.
\end{prop}
\begin{proof}
Let $M\in\x$ be an $R$-module as in the proposition, and set $m=\gd M$.
Then $\Omega^nM\cong\Omega^{n+t}M$ for some $n\ge m$ and $t>0$.
Without loss of generality, we may assume that $t>1$. Hence we have the following isomorphisms:
\begin{multline}\tag{\ref{pr1}.1}
\Omega^{-2}\Omega^mM
\cong\Omega^{-2}(\Omega^{m-n}(\Omega^nM))
\cong\Omega^{-2}\Omega^{m-n}(\Omega^{n+t}M)\\
\cong\Omega^{-2}\Omega^{m-n}\Omega^{t}(\Omega^nM)
\cong\Omega^{t-2}\Omega^{m-n}(\Omega^{n}M)
\cong\Omega^{t-2}(\Omega^{m}M)\in\x.
\end{multline}
Now the assertion follows from (\ref{pr1}.1) and \cite[Theorem 4.10]{DT}.
\end{proof}
The following is a direct consequence of Proposition \ref{pr1},
\cite[Theorem 1.6]{Av} and \cite[Theorem 1.2]{GP}.
\begin{cor}\label{c1}
Let $R$ be a Gorenstein local ring and let $\x$ be a resolving subcategory of $\md(R)$. Assume that $\x$ contains a module with bounded Betti numbers which is not maximal Cohen--Macaulay. Then $\x$ has infinite radius if either $\cod R\leq4$ or $\e(R)\leq11$.
\end{cor}
\section{Closedness under cosyzygies}
In this section, we study resolving subcategories which are closed under cosyzygies. We prove our second main result, Theorem \ref{tt}.
\begin{lem}\label{l3}
Let $\x$ be a resolving subcategory of $\md(R)$. The following are equivalent:
\begin{enumerate}[(i)]
	\item{$\Omega^{-1}\x\subseteq\x\subseteq\Omega(\md(R))$.}
	\item{$\Omega\x=\x$.}
\end{enumerate}	
\end{lem}
\begin{proof}
(i)$\Rightarrow$(ii).
As $\x$ is resolving, $\Omega\x\subseteq\x$. On the other hand, for all $M\in\x$, there exists an exact sequence $0\rightarrow M\rightarrow P\rightarrow\Omega^{-1}M\rightarrow0$, where $P$ is projective. As $\Omega^{-1}M\in\x$, $M\cong\Omega(\Omega^{-1}M)\in\Omega\x.$ Hence $\Omega\x=\x$.

(ii)$\Rightarrow$(i). For all $R$-module $M\in\x$, there exists an $R$-module $N\in\x$ such that $M\cong\Omega N$.
By \cite[Proposition 2.21]{AB}, there exists an exact sequence
$0\rightarrow Q\rightarrow\Omega^{-1}\Omega N\rightarrow N\rightarrow0$,
where $Q$ is projective. As $\x$ is resolving, it follows from the above exact sequence that $\Omega^{-1}M\cong\Omega^{-1}\Omega N\in\x$. Hence
$\Omega^{-1}\x\subseteq\x$.
\end{proof}
\begin{lem}\label{l4}
Let $\x$ be a subcategory of $\md(R)$ such that the following conditions hold:
\begin{enumerate}[(i)]
\item{$\x\subseteq\Omega(\md(R))$}
\item{$\x$ is closed under $\Omega^{-1}$ and kernel of epimorphisms.}
\item{$\x\cap\Pd(R)=\add(R)$.}
\end{enumerate}
Then $\x\subseteq\g(R)$.
\end{lem}
\begin{proof}
By (i), (ii), for all $R$-module $M\in\x$, we obtain the following long exact sequence
\begin{equation}\tag{\ref{l4}.1}
0\longrightarrow M\overset{d_0}{\longrightarrow} Q_0\overset{d_1}{\longrightarrow} Q_1\longrightarrow\cdots
\end{equation}
where $\coker(d_i)=\Omega^{-i-1}M$ for all $i\geq0$ and $Q_i$ is projective for all $i$. As $\Ext^1_R(\Omega^{-i}M,R)=0$ for all $i>0$, the exact sequence (\ref{l4}.1) induces the following exact sequence:
\begin{equation}\tag{\ref{l4}.2}
\cdots\longrightarrow (Q_1)^*\longrightarrow (Q_0)^*\longrightarrow M^*\longrightarrow 0.
\end{equation}
Dualizing the exact sequence (\ref{l4}.2), we get the following commutative diagram
$$\begin{CD}
0@>>> M@>>> Q_0 @>>> Q_1 @>>>\cdots  \\
&& @VVV@VV{\|\wr}V@VV{\|\wr}V \\
0@>>> M^{**}@>>> (Q_0)^{**}@>>> (Q_1)^{**}@>>>\cdots\\
\end{CD}
$$
It follows from the above diagram that $M\cong M^{**}$ and $\Ext^i_R(M^*,R)=0$ for all $i>0$. In other words, $M$ is $n$-torsion free for all $M\in\x$ and $n>0$.
Now it is enough to show that $\Ext^i_R(M,R)=0$ for all $i>0$. First we prove the following claim.

\textbf{Claim}. For every $R$-module $N\in\x$, if $\Ext^n_R(N,R)=0$ for some $n>0$, then $\Ext^{n+1}_R(N,R)=0$.

Set $Y=\Omega^{n-1}N$. Note that $\Omega^2Y\in\x$ and so it is $2$-torsion-free. Hence, by \cite[Proposition 2.21]{AB}, there exists an exact sequence
\begin{equation}\tag{\ref{l4}.3}
\sigma: 0\longrightarrow Z\longrightarrow\Omega^{-2}\Omega^2Y\longrightarrow Y\longrightarrow 0,
\end{equation}
where $\pd_R(Z)$ is finite. By (i) and (ii), $Z\in\x$. Hence by (iii), $Z$ is projective. As $\Ext^1(Y,R)=0$, $\Ext^1_R(Y,Z)=0$. In other words, $\sigma$ is a split exact sequence and so
$Y\approx\Omega^{-2}\Omega^2Y$. Note that $\Omega^{-1}\Omega^2 Y\in\x$ and so it is a first syzygy module by (i). Hence we have the following exact sequence: $0\rightarrow\Omega^{-1}\Omega^{2}Y\rightarrow P\rightarrow\Omega^{-2}\Omega^2Y\rightarrow0$, where $P$ is a projective $R$-module. Therefore we get the following isomorphisms:
\begin{equation}\tag{\ref{l4}.4}
\Ext^{n+1}_R(N,R)
\cong\Ext^2_R(Y,R)
\cong\Ext^2_R(\Omega^{-2}\Omega^{2}Y,R)
\cong\Ext^1_R(\Omega^{-1}\Omega^{2}Y,R)=0,
\end{equation}
which proves our claim. Now let $M\in\x$. By (i), $M$ is a first syzygy module and we have the following exact sequence
$0\rightarrow M\rightarrow Q\rightarrow\Omega^{-1}M\rightarrow0$, where $Q$ is a projective $R$-module. By (ii), $\Omega^{-1}M\in\x$. As $\Ext^1_R(\Omega^{-1}M,R)=0$, $\Ext^i_R(M,R)\cong\Ext^{i+1}_R(\Omega^{-1}M,R)=0$ for all $i>0$ by the claim. Hence $\x\subseteq\g(R)$.
\end{proof}
\begin{thm}\label{t4}
Let $\x$ be a resolving subcategory of $\md(R)$	with $\x=\Omega\x$. Then $\x\subseteq\g(R)$.
\end{thm}
\begin{proof}
By our assumption, for every $R$-module $M\in\x$ and every positive integer number $n$ there exists an $R$-module $N\in\x$ such that $M\cong\Omega^{n}N$. Therefore $\x\cap\Pd(R)=\add(R)$ and so the assertion is clear by Lemmas  \ref{l3} and \ref{l4}.	
\end{proof}	
Let $R$ be a Golod local ring which is not a hypersurface (e.g. a Cohen-Macaulay non-Gorenstein local ring with minimal multiplicity). In \cite[Example 3.5]{AvM}, it is shown that $\add(R)=\g(R)$.
The following can be viewed as a generalization of this result.
\begin{cor}\label{c3}
Let $R$ be a Golod local ring and let $\x$ be a non-trivial resolving subcategory of $\md(R)$ (i.e. $\add(R)\subsetneq\x\subsetneq\md(R)$). Then $\x=\Omega\x$ if and only if $R$ is a hypersurface and $\x\subseteq\g(R)$.	
\end{cor}
\begin{proof}
If $\x=\Omega\x$, then by Theorem \ref{t4} $\add(R)\subsetneq\x\subseteq\g(R)$. Hence, by \cite[Example 3.5]{AvM}, $R$ is a hypersurface.
On the other hand, if $R$ is a hypersurface and $\x\subseteq\g(R)$ then $\Omega^{-1}M\cong\Omega^{2}\Omega^{-1}M\cong\Omega M$ for all $M\in\x$ by \cite[Theorem 6.1]{E}. Now the assertion follows from Lemma \ref{l3}.
\end{proof}
Here is an example of a subcategory which is closed under cosyzygies.
\begin{ex}\label{eee}
Let $S$ be any subset of $\Spec R$.
Consider the subcategory
$$
\NF^{-1}(S)=\{M\in\md R\mid\NF(M)\subseteq S\}
$$
of $\md R$.
Then it is easy to see that $\NF^{-1}(S)$ is a resolving subcategory closed under $\Tr$.
Hence it is also closed under $\Omega^{-1}$.
In particular, $\md_0R=\NF^{-1}(\{\fm\})$ is a resolving subcategory of $\md R$ closed under $\Omega^{-1}$.
\end{ex}
\begin{prop}\label{p}
Let $R$ be a local ring and let $\x$ be a resolving subcategory of $\md(R)$ such that $\Omega^{-1}\x\subseteq\x\cap \Omega(\md(R))$. Then either of the following holds.
\begin{enumerate}[(i)]
\item{$\x\subseteq\g(R)$.}
\item{$R$ is a Gorenstein ring of dimension $1$}.	
\end{enumerate}
\end{prop}
\begin{proof}
Assume that $R$ is not a Gorenstein ring of dimension $1$. We want to show that $\x\subseteq\g(R)$. By Theorem (\ref{t4}), it is enough to show that $\x\subseteq\Omega\x$. Let $M\in\x$ and consider the following exact sequences
$0\rightarrow\Omega M\rightarrow F\rightarrow M\rightarrow0$ and
$0\rightarrow\Omega M\rightarrow F_1\rightarrow \Omega^{-1}\Omega M\rightarrow0$. We make the following pushout diagram
$$\begin{CD}
&&&&&&0 &&0 \\
&&&&&& @VVV@VVV\\
\ \ &&&& 0@>>>\Omega M @>>>F@>>> M @>>>0&  \\
&&&&&& @VVV @VVV @VV{\parallel}V \\
\ \  &&&&0 @>>>F_1 @>>> T @>>>M @>>>0&\\
&&&&&&@VVV@VVV
\ \ &&&&&&\\
&&&&&&\Omega^{-1}\Omega M &=& \Omega^{-1}\Omega M\\
&&&&&&@VVV@VVV\\
&&&&&&0&&0
\end{CD}$$\\
Note that $\Ext^1_R(\Omega^{-1}\Omega M,R)=0$.
Hence the second column in the above diagram splits, and we get an exact sequence
$$0\longrightarrow F_1\longrightarrow F\oplus\Omega^{-1}\Omega^1M\longrightarrow M\longrightarrow0.$$
As $\Omega^{-1}\Omega M\in\x$, it is a first syzygy module by our assumption and so we have the following exact sequence $0\rightarrow\Omega^{-1}\Omega M\rightarrow F'\rightarrow \Omega^{-2}\Omega M\rightarrow0$, which induces the following exact sequence:
$0\rightarrow F\oplus\Omega^{-1}\Omega M\rightarrow F\oplus F'\rightarrow\Omega^{-2}\Omega M\rightarrow0$.
Thus the following pushout diagram is obtained:
$$\begin{CD}
&&&&&&&&0 &&0 \\
&&&&&&&& @VVV@VVV\\
\ \ &&&& 0@>>>F_1 @>>>F\oplus\Omega^{-1}\Omega M@>>> M @>>>0&  \\
&&&&&& @VV{\parallel}V @VVV @VVV \\
\ \  &&&&0 @>>>F_1 @>>> F\oplus F' @>>> X @>>>0&\\
&&&&&&&&@VVV@VVV\\
&&&&&&&&\Omega^{-2}\Omega M &=& \Omega^{-2}\Omega M\\
&&&&&&&&@VVV@VVV\\
&&&&&&&&0&&0
\end{CD}$$\\
It follows from the above diagram that $X\in\x$ and $\pd_R(X)\leq1$. Hence it is enough to show that $X$ is free. To do this, we prove that $\x\cap\Pd(R)\subseteq\add(R)$.
If $\depth R=0$, then the assertion follows from Auslander-Buchsbaum formula, so we may assume that $\depth R>0$.
Assume contrarily that $\x\cap\Pd(R)\nsubseteq\add(R)$. Hence there exists an $R$-module $N\in\x$ with $\pd_R(N)=1$.
By \cite[Lemma 4.6]{DT}, we may assume that $N$ is locally free.
Therefore $\Ext^1_R(N,R)$ is a nonzero $R$-module of finite length. Let $\sigma\in\soc(\Ext^1_R(N,R))$ be a non-zero element. Hence we have the following non-split exact sequence:
\begin{equation}
\sigma: 0\rightarrow R\rightarrow L\rightarrow N\rightarrow0,
\end{equation}
which induces the following exact sequence: $0\rightarrow k\rightarrow\Ext^1_R(N,R)\rightarrow\Ext^1_R(L,R)\rightarrow0$. Therefore $\ell(\Ext^1_R(L,R))<\ell(\Ext^1_R(N,R))$. Replacing $N$ by
$L$ and repeating this process if $\ell(\Ext^1_R(L,R))>0$, we can assume that
$\Ext^1_R(L,R)=0$. Therefore $\Ext^1_R(N,R)\cong k$. As $\pd_R(N)=1$, $\Tr N\cong k$. Hence $\Tr k\cong N\in\x$. By our assumption, $\Tr\Omega^ik\cong\Omega^{-i}\Tr k\in\x\cap\Omega(\md(R))$ for all $i>0$. Therefore $\Ext^i_R(k,R)=0$ for all $i>1$ which is a contradiction because $R$ is not a Gorenstein ring of dimension one.
\end{proof}	
In the following example, it is shown that over a one dimensional Gorenstein local ring $R$ there exists a resolving subcategory of $\md(R)$ which satisfies in the assumption of Proposition \ref{p}, but not contained in $\g(R)$.
\begin{ex}
Let $(R,\fm,k)$ be a Gorenstein local ring of dimension $1$. Set
$\x=\md_0(R)$. Then the following hold.
\begin{enumerate}[(i)]
\item{$\x$ is resolving.}
\item{$\Omega^{-1}\x$ is contained in $\x\cap\Omega(\md(R))$.}
\item{$\x$ is not contained in $\g(R).$}
\end{enumerate}
\end{ex}
\begin{proof}
Part (i) is stated in Example \ref{eee}. For all $R$-modules $M$,
$\Omega\Tr M\in\g(R)$, because $R$ is Gorenstein of dimension $1$. Hence, by Lemma \ref{le} and Proposition \ref{G}(i), $\Omega^{-1}M=\Tr\Omega\Tr M\in\g(R)$ and so (ii) is clear.
Also $k$ is an element of $\x$ which is not contained in $\g(R)$.
\end{proof}
Now we can prove the main result of this section.
\begin{thm}\label{t5}
	Let $(R,\fm,k)$ be a local ring and let $\x$ be a resolving subcategory of $\md(R)$. If $\Omega^{-n}\x\subseteq\x$ for some $1\leq n\leq\depth R+1$, then either $\x\subseteq\g(R)$ or $\Omega^{n-1}k\in\x$ holds.
In particular, if $\Omega^{-1}\x\subseteq\x$, then either $\x\subseteq\g(R)$ or $k\in\x$ holds.
\end{thm}
\begin{proof}
We begin with the case $\depth R=0$; then $n=1$.
First we consider the case $\Ext^1_R(\x,R)=0$. Then
$\Ext^{>0}_R(\x,R)=0$. Let $M\in\x$ and $i\ge0$. By Lemma \ref{le} and Proposition \ref{p1}(ii), there is an exact sequence
\begin{equation}\tag{\ref{t5}.1}
0\rightarrow  E \rightarrow\Omega^{-i}M\rightarrow\Omega\Omega^{-(i+1)} M\rightarrow 0,
\end{equation}
where $E:=\Ext^{i+1}_R(\Tr M,R)$. As $\x$ is closed under cosyzygies, it follows from the exact sequence (\ref{t5}.1) that $E\in\x$. Hence $\Ext^{>0}_R(E,R)=0$. By \cite[Theorem 2.8]{AB}, there is an exact sequence
\begin{equation}\tag{\ref{t5}.2}
\Tor_{i+1}^R(\Tr M,R) \rightarrow E^*\rightarrow\Ext^2_R(\Omega^{-(i+1)} M,R).
\end{equation}
As $\Omega^{-(i+1)} M\in\x$, $\Ext^2_R(\Omega^{-(i+1)} M,R)=0$. It follows from the exact sequence (\ref{t5}.2) that $E^*=0$.
Hence $E$ has infinite grade, and $E=0$. Thus $\Ext^{>0}_R(\Tr\x,R)=0$, and $\x$ is contained in $\g(R)$.

Next, let us consider the case $\Ext^1_R(\x,R)\ne0$. We claim that
$\Ext^2_R(\x,R)\ne0$. Indeed, suppose $\Ext^2_R(\x,R)=0$. Since $\Ext^1_R(\x,R)\ne0$, we find a module $M\in\x$ such that $\Ext^1_R(M,R)\ne0$. By lemma \ref{l6}, we may assume $M$ is locally free on the punctured spectrum. Setting $M_0=M$ and
$M_{i+1}=\Ext^1_R(\Tr M_i,R)$ for $i\ge0$, we have an exact sequence
\begin{equation}\tag{\ref{t5}.3}
0\rightarrow M_{i+1}\rightarrow M_i\rightarrow\Omega\Omega^{-1}M_i\rightarrow 0.
\end{equation}
Inductively, the second and third modules are in $\x$, and so is the
first. We get a descending chain
$\cdots\subseteq M_3\subseteq M_2\subseteq M_1$
of modules in $\x$. Since all of these are of finite length, for some $n>0$ we have $M_n=M_{n+1}$. This means that $\Omega\Omega^{-1}M_n=0$, and
$M_n=R^{\oplus e}$ for some $e\ge0$ since $\depth(R)=0$. The exact sequence (\ref{t5}.3) induces the following exact sequence:
$$\Ext^2_R(\Omega^{-1}M_i,R)\rightarrow\Ext^1_R(M_i,R)\rightarrow \Ext^1_R(M_{i+1},R),$$
and the first Ext module vanishes since $\Ext^2_R(\x,R)=0$. Since
$\Ext^1_R(M_0,R)\ne0$, we inductively see that $\Ext^1_R(M_n,R)\ne0$ which is a contradiction, because $M_n$ is a free $R$-module.

Thus we find a module $M\in\x$ with $\Ext^2_R(M,R)\ne0$. Again, by Lemma \ref{l6}, we may assume
$M$ is locally free on the punctured spectrum. Hence $\Ext^2_R(M,R)$ has
finite length, and it contains a nonzero socle element, which is
regarded as an exact sequence
$0\rightarrow R\rightarrow Y\rightarrow\Omega M\rightarrow 0.$
Since $M\in\x$, $Y$ is also in $\x$. Dualizing this, we get an exact sequence
\begin{equation}\tag{\ref{t5}.4}
0\rightarrow k\rightarrow\Tr\Omega M\rightarrow\Tr Y\rightarrow 0.
\end{equation}
By rotating the exact sequence (\ref{t5}.4), we get the following exact sequence
$$0 \rightarrow\Omega\Tr\Omega M\rightarrow\Omega\Tr Y\rightarrow k\rightarrow 0.$$
Dualizing the above exact sequence, we have an exact sequence
$$0\rightarrow k^{\oplus r}\rightarrow(\Omega\Tr Y)^*\xrightarrow{a} (\Omega\Tr\Omega X)^*\rightarrow\Ext^1_R(k,R)\xrightarrow{b}\Ext^2_R(\Tr Y,R),$$
where $r=r(R)$, the type of $R$. The exact sequence (\ref{t5}.4) induces an exact sequence
$$0=\Ext^1_R(\Tr\Omega M,R)\rightarrow\Ext^1_R(k,R)\xrightarrow{b} \Ext^2_R(\Tr Y,R),$$
whence the map $b$ is injective. Therefore the map $a$ is surjective, and
we obtain an exact sequence
$$0\rightarrow k^{\oplus r}\rightarrow(\Omega\Tr Y)^*\xrightarrow{a} (\Omega\Tr\Omega M)^*\rightarrow0.$$
Since the second and third modules are in $\x$, so is the first, and so is $k$.
Consequently, the proof of the theorem is completed in the case where $\depth R=0$.

Now we assume that $\depth R>0$.
First we claim the following:\\
\textbf{Claim I}. If $\Omega^{n-1}k\notin\x$, then $\Ext^{n+1}_R(\x,R)=0$.\\
Proof of Claim. Assume contrarily that  $\Ext^{n+1}_R(\x,R)\neq0$.
Then, by Lemma \ref{l6}, there exists an $R$-module $M\in\x\cap\md_0(R)$ such that $\Ext^{n+1}_R(M,R)\neq0$.
Hence $\Ext^{n+1}_R(M,R)$ has finite length, and it contains a nonzero socle element, which is
regarded as an exact sequence
$0\rightarrow R\rightarrow N\rightarrow\Omega^n M\rightarrow 0.$
Since $M\in\x$, $N$ is also in $\x$. Dualizing this, we get an exact sequence
\begin{equation}\tag{\ref{t5}.5}
0\rightarrow k\rightarrow\Tr\Omega^n M\rightarrow\Tr N\rightarrow 0.
\end{equation}
By rotating the exact sequence (\ref{t5}.5), we get the following exact sequence
$$0 \rightarrow\Omega\Tr\Omega^n M\rightarrow\Omega\Tr N\rightarrow k\rightarrow 0.$$
The above exact sequence induces the following exact sequence:
\begin{equation}\tag{\ref{t5}.6}
0 \rightarrow\Omega^n\Tr\Omega^n M\rightarrow\Omega^n\Tr N\rightarrow \Omega^{n-1}k\rightarrow 0.
\end{equation}
The exact sequence (\ref{t5}.6) induces the following exact sequence.
\begin{equation}\tag{\ref{t5}.7}
(\Omega^n\Tr N)^*\xrightarrow{a}(\Omega^n\Tr\Omega^n M)^*\rightarrow\Tr\Omega^{n-1}k\xrightarrow{b}\Omega^{-n}N\rightarrow\Omega^{-n}\Omega^n M\rightarrow 0.
\end{equation}
Dualizing the exact sequence (\ref{t5}.6), we have an exact sequence
\begin{equation}\tag{\ref{t5}.8}
(\Omega^n\Tr N)^*\xrightarrow{a} (\Omega^n\Tr\Omega^n X)^*\rightarrow\Ext^n_R(k,R)\xrightarrow{c}\Ext^{n+1}_R(\Tr N,R).
\end{equation}
The exact sequence (\ref{t5}.5) induces an exact sequence
\begin{equation}\tag{\ref{t5}.9}
\Ext^{n}_R(\Tr\Omega^n M,R)\rightarrow\Ext^n_R(k,R)\xrightarrow{c} \Ext^{n+1}_R(\Tr N,R).
\end{equation}
As $M$ is locally free on the punctured spectrum and $\depth R\geq n-1$, $\Omega^n M$ is $n$-torsion free by Lemma \ref{t7}.
Hence, $\Ext^n_R(\Tr\Omega^n M,R)=0$ and so the map $c$ is injective by the exact sequence (\ref{t5}.9). Therefore the map $a$ is surjective by the exact sequence (\ref{t5}.8), and so the map $b$ is injective by the exact sequence (\ref{t5}.7). Hence, we get the following exact sequence
\begin{equation}\tag{\ref{t5}.10}
0\rightarrow\Tr\Omega^{n-1}k\rightarrow\Omega^{-n}N\rightarrow\Omega^{-n}\Omega^n M\rightarrow 0.
\end{equation}
It follows from the exact sequence (\ref{t5}.10) that $\Tr\Omega^{n-1}k\in\x$. Set $t=\depth R\geq n-1$ and consider the following exact sequence
\begin{equation}\tag{\ref{t5}.11}
0\rightarrow\Ext^{t}_R(k,R)\rightarrow\Tr\Omega^{t-1}k\rightarrow\Omega\Tr\Omega^{t}k\rightarrow0.
\end{equation}
If $t=n-1$, then $\Ext^i_R(k,R)=0$ for all $i<n-1$ and it is easy to see that $\pd_R(\Tr\Omega^{n-2}k)\leq n-1$. Hence
the exact sequence (\ref{t5}.11) induces the following exact sequence
\begin{equation}\tag{\ref{t5}.12}
0\rightarrow\Omega^{n-1}\Ext^{n-1}_R(k,R)\rightarrow F\rightarrow\Omega^n\Tr\Omega^{n-1}k\rightarrow 0,
\end{equation}
where $F$ is free. As $\Tr\Omega^{n-1}k\in\x$, it follows from the exact sequence (\ref{t5}.12) that $\Omega^{n-1}k\in\x$ which is a contradiction.
Now let $t\geq n$. As $\Ext^i_R(k,R)=0$ for $i<n$, it is easy to see that $\Tr\Omega^{i-1}k\approx\Omega\Tr\Omega^{i}k$ for $0<i<n$. In particular, $\Tr\Omega^ik\approx\Omega^{n-i-1}\Tr\Omega^{n-1}k\in\x$ for $0\leq i\leq n-1$. As $\x$ is closed under $\Omega^{-n}$ and $\Tr\Omega^{n-1}k\in\x$, it is easy to see that $\Tr\Omega^ik\in\x$ for all $i\geq0$.
It follows from the exact sequence (\ref{t5}.11) that $k\in\x$ which is a contradiction. Thus, the proof of the claim is
completed.

Next we claim the following.\\
\textbf{Claim II}. If $\Omega^{n-1}k\notin\x$, then $\gd_R(M)\leq n$ for all $M\in\x\cap\md_0(R)$.\\
Proof of Claim. Let $M\in\x\cap\md_0(R)$ and set $N=\Omega^nM$. We want to prove that $\gd_R(N)=0$. By claim I, it is enough to show that $\Ext^i_R(\Tr N,R)=0$ for all $i>0$. As $\depth R\geq n-1$,
$\Ext^i_R(\Tr N,R)=0$ for $1\leq i\leq n$ by Lemma \ref{t7}. Therefore we obtain the following exact sequences:
$$\begin{CD}
\ \ &&&& 0@>>> N @>>>F_1@>>> \Omega^{-1}N @>>>0&  \\
\ \  &&&&0@>>>\Omega^{-1}N @>>> F_2 @>>>\Omega^{-2}N @>>>0&\\
&&&&&& \vdots&&\vdots&&\vdots\\
\ \  &&&&0@>>>\Omega^{-(n-1)}N @>>> F_n @>>>\Omega^{-n}N @>>>0&\\
\end{CD}$$\\
As $\Omega^{-n}N\in\x$, it follows from the above exact sequences that
$\Omega^{-i}N\in\x$ for $1\leq i\leq n$. Since $\x$ is closed under $\Omega^{-n}$, it is easy to see that $\Omega^{-i}N\in\x$ for all $i>0$.
Assume contrarily that $\Ext^i_R(\Tr N,R)\neq0$ for some $i>n$ and set
$t=\inf\{i>0\mid\Ext^i_R(\Tr N,R)\neq0\}>n$.
Now consider the following exact sequence:
\begin{equation}\tag{\ref{t5}.13}
0\rightarrow\Ext^t_R(\Tr N,R)\rightarrow\Omega^{-(t-1)}N\rightarrow\Omega\Omega^{-t}N\rightarrow0.
\end{equation}
The exact sequence (\ref{t5}.13) induces the following exact sequence:
\begin{equation}\tag{\ref{t5}.14}
\cdots\rightarrow\Ext^j_R(\Omega^{-(t-1)}N,R)\rightarrow\Ext^j_R(\Ext^t_R(\Tr N,R),R)\rightarrow\Ext^{j+2}_R(\Omega^{-t}N,R)\rightarrow\cdots.
\end{equation}
Note that $N$ is $(t-1)$-torsion free.
Hence $N\approx\Omega^{t-1}\Omega^{-(t-1)}N$.
By Claim I, $\Ext^i_R(N,R)=0$ for all $i>0$. Now it is easy to see that $\Ext^i_R(\Omega^{-(t-1)}N,R)=0$ for all $i>0$. As $\Omega^{-t}N\in\x$, by Claim I, $\Ext^i_R(\Omega^{-t}N,R)=0$ for all $i>n$. Therefore, by the exact sequence (\ref{t5}.14), we get
$\Ext^j_R(\Ext^t_R(\Tr N,R),R)=0$ for $j\geq\max\{1, n-1\}$. On the other hand,
$\Ext^t_R(\Tr N,R)$ has finite length and so $\gr_R(\Ext^t_R(\Tr N,R))\geq\depth R\geq\max\{1,n-1\}$. Therefore $\Ext^j_R(\Ext^t_R(\Tr N,R),R)=0$ for all $j\geq0$ and so $\Ext^t_R(\Tr N,R)=0$ which is a contradiction. Thus, the proof of the claim is completed.

\textbf{Claim III}. If $\Ext^1_R(\x,R)=0$, then $\x\subseteq\g(R)$.\\
Proof of Claim. Assume contrarily that there exists an $R$-module $M\in\x$ which is not totally reflexive. Hence $\Ext^m_R(\Tr M,R)\neq0$ for some $m>0$. Consider the following exact sequence
\begin{equation}\tag{\ref{t5}.15}
0\rightarrow\Ext^m_R(\Tr M,R)\rightarrow\Omega^{-(m-1)}M\rightarrow\Omega\Omega^{-m}M\rightarrow0.
\end{equation}
There exists a positive integer $r$ such that $(r-1)n< m\leq rn$. As $\Ext^i_R(M,R)=0$ for all $i>0$, it is easy to see that $$\Omega^{-m}M\approx\Omega^{-(m+1)}\Omega M\approx\cdots\approx\Omega^{-(rn)}\Omega^{rn-m}M\in\x.$$
Similarly, one can see that $\Omega^{-(m-1)}M\in\x$. It follows from the exact sequence (\ref{t5}.15) that $\Ext^m_R(\Tr M,R)\in\x$ and so $\Ext^i_R(\Ext^m_R(\Tr M,R),R)=0$ for all $i>0$. By \cite[Theorem 2.8]{AB}, we have the following exact sequence
$$0=\Tor_m^R(\Tr M,R)\rightarrow(\Ext^m_R(\Tr M,R))^*\rightarrow\Ext^2_R(\Omega^{-m}M,R).$$
As $\Omega^{-m}M\in\x$, $\Ext^2_R(\Omega^{-m}M,R)=0$ and so $(\Ext^m_R(\Tr M,R))^*=0$. Therefore $\Ext^m_R(\Tr M,R)=0$ which is a contradiction and the proof of the claim is completed.

Now Assume contrarily that $\x\nsubseteq\g(R)$ and $\Omega^{n-1}k\notin\x$. By Claim III and Lemma \ref{l6}, $\Ext^1_R(M,R)\neq0$ for some $M\in\x\cap\md_0(R)$. It follows from the Claim II that $\gd_R(M)<\infty$. So we may assume that
$\gd_R(M)=1$. Now $\Ext^1_R(M,R)$ is a nonzero $R$-module of finite length, and we can choose a socle
element $0\neq\sigma\in\Ext^1_R(M,R)$. It can be represented as a short exact sequence:
$$\sigma: 0\rightarrow R\rightarrow M_1\rightarrow M\rightarrow0.$$
The $M_1$ belongs to $\x$, is locally free on the punctured spectrum of $R$ and has Gorenstein dimension at most $1$. The above exact sequence induces the following exact sequence
$$0\rightarrow k\rightarrow\Ext^1_R(M,R)\rightarrow\Ext^1_R(M_1,R)\rightarrow0.$$
This implies that $\ell(\Ext^1_R(M_1,R))=\ell(\Ext^1_R(M,R))-1$.
Replacing $M$ by $M_1$ and
repeating this process if $\ell(\Ext^1_R(M_1,R))>0$, we can assume that $\Ext^1_R(M,R)\cong k$. There are exact sequences
$0\rightarrow\Omega M\rightarrow F\rightarrow M\rightarrow 0$ and $0\rightarrow\Omega M\rightarrow F'\rightarrow\Omega^{-1}\Omega M\rightarrow 0$, where the latter is possible as
$\Omega M$ is a totally reflexive module. We make the following pushout diagram.
$$\begin{CD}
&&&&&&0 &&0 \\
&&&&&& @VVV@VVV\\
\ \ &&&& 0@>>>\Omega M @>>>F@>>> M @>>>0&  \\
&&&&&& @VVV @VVV @VV{\parallel}V \\
\ \  &&&&0 @>>>F' @>>> T @>>>M @>>>0&\\
&&&&&&@VVV@VVV
\ \ &&&&&&\\
&&&&&&\Omega^{-1}\Omega M &=& \Omega^{-1}\Omega M\\
&&&&&&@VVV@VVV\\
&&&&&&0&&0
\end{CD}$$\\
Note that $\Ext^1_R(\Omega^{-1}\Omega M,R)=0$.
Hence the second column in the above diagram splits, and we get an exact sequence
$0\longrightarrow F'\longrightarrow F\oplus\Omega^{-1}\Omega^1M\longrightarrow M\longrightarrow0.$
The exact sequence $0\rightarrow\Omega^{-1}\Omega M\rightarrow F''\rightarrow \Omega^{-2}\Omega M\rightarrow0$ induces the following exact sequence:
$0\rightarrow F\oplus\Omega^{-1}\Omega M\rightarrow F\oplus F'\rightarrow\Omega^{-2}\Omega M\rightarrow0$.
Thus the following pushout diagram is obtained:
$$\begin{CD}
&&&&&&&&0 &&0 \\
&&&&&&&& @VVV@VVV\\
\ \ &&&& 0@>>>F' @>>>F\oplus\Omega^{-1}\Omega M@>>> M @>>>0&  \\
&&&&&& @VV{\parallel}V @VVV @VVV \\
\ \  &&&&0 @>>>F' @>>> F\oplus F'' @>>> X @>>>0&\\
&&&&&&&&@VVV@VVV\\
&&&&&&&&\Omega^{-2}\Omega M &=& \Omega^{-2}\Omega M\\
&&&&&&&&@VVV@VVV\\
&&&&&&&&0&&0
\end{CD}$$\\
As $\gd_R(\Omega M)=0$ and $\x$ is closed under $\Omega^{-n}$,
it is easy to see that $\Omega^{-2}\Omega M\in\x$.
The second column shows
that $X\in\x$ and $\Ext^1_R(X,R)\cong\Ext^1_R(M,R)\cong k$.
Hence, by the second row, we conclude that $\pd_R(X)=1$. Therefore
$\Tr X=k$ and so $\Tr k=X\in\x$. In particular, $\Tr\Omega^nk=\Omega^{-n}\Tr k\in\x$. By Claim II, $\gd_R(\Tr\Omega^nk)\leq n$. On the other hand, $\Ext^i_R(\Tr\Omega^nk,R)=0$ for $1\leq i\leq n$ by Lemma \ref{t7} and so $\gd_R(\Tr\Omega^nk)=0$ by Proposition \ref{G}(iii). Hence, by Proposition \ref{G}, $n-1\leq\depth R=\gd_R(k)\leq n$ and $R$ is Gorenstein. Now consider the following exact sequence
$$0\rightarrow\Ext^t_R(k,R)\rightarrow\Tr\Omega^{t-1}k\rightarrow\Omega\Tr\Omega^{t}k\rightarrow0,$$
where $t=\depth R$. As $\Ext^i_R(k,R)=0$ for $i<t$,
$\Omega^{t-1}\Tr\Omega^{t-1}k\cong\Tr k\in\x$.
The above exact sequence induces the following exact sequence:
$$0\rightarrow\Omega^{t-1}\Ext^t_R(k,R)\rightarrow F\oplus\Omega^{t-1}\Tr\Omega^{t-1}k\rightarrow\Omega^t\Tr\Omega^{t}k\rightarrow0,$$
where $F$ is a free $R$-module. If $t=n$, then $\Tr\Omega^{t}k\in\x$. Also if $t=n-1$, then $\Ext^n_R(k,R)=0$ and so $\Tr\Omega^{t}k\approx\Omega\Tr\Omega^{n}k\in\x$. Hence, by the above exact sequence we get that $\Omega^{t-1}k\in\x$ which is a contradiction.
\end{proof}
The following example should be compared with Theorem \ref{t5}.
\begin{ex}\label{e}
\begin{enumerate}[(i)]
\item
By Example \ref{eee}, $\x=\md_0R$ is a resolving subcategory of $\md R$ with $\Omega^{-1}\x\subseteq\x$.
It is clear that $k\in\x$, while $\x\nsubseteq\g(R)$ unless $R$ is Artinian Gorenstein.
\item
Let $A$ be a local ring with residue field $k$, and let $R=A[[x]]/(x^2)$ with $x$ an indeterminate over $A$.
Note that $R/(x)$ is a totally reflexive $R$-module.
Let $\x$ be the thick closure of $R/(x)$, that is, the smallest thick subcategory of $\g(R)$ containing $R/(x)$.
As there is an exact sequence $0\to R/(x)\to R\to R/(x)\to0$, $\x$ contains $R$.
Hence $\x$ is a resolving subcategory of $\md R$ with $\x\subseteq\g(R)$.
Consider the subcategory
$
\y:=\{X\in\x\mid\Tr X\in\x\}.
$
This contains $R/(x)$ as $\Tr(R/(x))=R/(x)$.
Let $0\to L\to M\to N\to0$ be an exact sequence in $\g(R)$.
Then $\Ext_R^1(N,R)=0$, and the induced sequence $0\to\Tr N\to\Tr M\to\Tr L\to0$ is exact by Proposition \ref{p1}(i).
It is easy to see from this that $\y$ is a thick subcategory of $\g(R)$.
Hence $\y=\x$, and $\x$ is closed under $\Tr$.
Therefore we obtain $\Omega^{-1}\x\subseteq\x$.
However, $k\notin\x$ unless $R$ is Artinian Gorenstein.
\end{enumerate}
\end{ex}
Theorem \ref{t5} immediately yields:
\begin{cor}
Let $R$ be a Cohen--Macaulay local ring of dimension $d>0$.
Let $\x\subseteq\cm(R)$ be a resolving subcategory closed under cosyzygies. Then $\x$ is contained in $\g(R)$. In particular, $R$ is Gorenstein if and only if $\cm(R)$ is closed under cosyzygies.
\end{cor}
\begin{cor}
Let $(R,\fm,k)$ be a local ring that is not Artinian Gorenstein. Then $\g(R)$ is the maximum resolving subcategory closed under transposes and not containing $k$.
\end{cor}
\begin{rmk}
In the above corollary, the condition ``not containing $k$" cannot be removed.
In fact, $\md_0R$ is a resolving subcategory closed under transposes, but it is not contained in $\g(R)$.
\end{rmk}
The following is an immediate consequence of Theorem \ref{t5} and \cite[Theorem 4.4]{DT}
\begin{cor}
Let $R$	be a local ring and let $\x\subsetneq\md(R)$ be a resolving subcategory. Assume that $\x$ is closed under transpose. If $\rad(\x)<\infty$, then $\x\subseteq\g(R)$.
\end{cor}
\begin{cor}\label{c4}
Let $R$	be a local ring and let $\x\subsetneq\md(R)$ be a resolving subcategory of finite radius. Assume that $\Omega^{-n}\x\subseteq\x$ for some $n$, $1\leq n\leq\depth R+1$.
Then $\x\subseteq\g(R)$.
\end{cor}
\begin{proof}
Assume contrarily that $\x\nsubseteq\g(R)$.
By Theorem \ref{t5}, $\Omega^{n-1}k\in\x$.
If $n=1$, then we get a contradiction by \cite[Theorem 4.4]{DT}. So assume that $n>1$.
It follows from Proposition \ref{t2} that
$R$ is Cohen--Macaulay of dimension $n-1$.
By Lemma \ref{l2}, $\x\subseteq\cm(R)$.
Set $M=\Omega^{n-1}k\in\x$. If $\Ext^1_R(M,R)=0$, then $R$ is Gorenstein and so $\x\subseteq\cm(R)=\g(R)$. So let $\Ext^1_R(M,R)\neq0$.
As $M$ is locally free on the punctured spectrum, $\Omega M$ is $n$-torsion free by Lemma \ref{t7}. In particular, $\Omega M$ is reflexive. Also, $\Omega^{-i}\Omega M\cong\Omega\Omega^{-(i+1)}\Omega M$ for $0\leq i\leq n-1$. As $\x$ is closed under $\Omega^{-n}$, it is easy to see that $\Omega^{-2}\Omega M\in\x$. It follows from Proposition \ref{p2} that $\rad(\x)=\infty$, which is a contradiction.
\end{proof}
The following result is an immediate consequence of Theorem \ref{t5}.
\begin{cor}
Let $R$ be a local ring which is not Artinian Gorenstein.
Consider the following subcategory of $\md R$.
$$\x=\{M\in\md(R)\mid\Ext^i_R(M,R)=0 \text{ for all } i>0\}.$$
Then $\x$ is closed under cosyzygies if and only if $\x=\g(R)$.	
\end{cor}
\begin{cor}\label{t6}
Let $(R,\fm,k)$ be a local ring and let $\x$ be a resolving subcategory of $\md(R)$.
Assume that $\Omega^{-n}\x\subseteq\x\subseteq\Omega^n(\md(R))$ for some $n$, $1\leq n\leq\depth R+1$, then $\x\subseteq\g(R).$	
\end{cor}
\begin{proof}
If $n=1$, then the assertion follows from Theorem \ref{t4} and Lemma \ref{l3}. So assume that $n>1$.	
It follows from Theorem \ref{t5} that either $\Omega^{n-1}k\in\x$ or $\x\subseteq\g(R)$. Assume that $\Omega^{n-1}k\in\x$. Hence $\Omega^{n-1}k$
is an $n$-th syzygy module. As $\depth R\geq n-1$, by Lemma \ref{t7}, $\Omega^{n-1}k$ is $n$-torsion free and so
$\Omega^{-i}\Omega^{n-1}k\cong\Omega\Omega^{-(i+1)}\Omega^{n-1}k$ for $0\leq i\leq n-1$. Since $\x$ is closed under $\Omega^{-n}$, it is easy to see that $\Omega^{-i}\Omega^{n-1}k\in\x$ for all $i\geq0$. Now consider the following exact sequence:
\begin{equation}\tag{\ref{t6}.1}
0\rightarrow\Ext^j_R(\Tr\Omega^{n-1}k,R)\rightarrow\Omega^{-(j-1)}\Omega^{n-1}k\rightarrow\Omega\Omega^{-j}\Omega^{n-1}k\rightarrow0.
\end{equation}
It follows from the exact sequence (\ref{t6}.1) that  $\Ext^j_R(\Tr\Omega^{n-1}k,R)\in\x$ for all $j>0$. Note that $\x\subseteq\Omega^n(\md(R))$ for some $n>0$. As $\ell(\Ext^j_R(\Tr\Omega^{n-1}k,R))<\infty$ for all $j>0$, $\Ext^j_R(\Tr\Omega^{n-1}k,R)=0$ for all $j>0$. In other words, $\Omega^{n-1}k$ is $i$-th syzygy module for all $i\geq0$ and so $\depth R=n-1$. By Lemma \ref{t7}, $\Omega^nk$ is $n$-torsion free. As $n>1$, by \cite[Proposition 2.21]{AB}, there exists an exact sequence $0\rightarrow L\rightarrow\Omega^{-2}\Omega^2(\Omega^{n-2}k)\rightarrow\Omega^{n-2}k\rightarrow0$ with $\pd_R(L)<\infty$. Rotating this gives an exact sequence
\begin{equation}\tag{\ref{t6}.2}
0\rightarrow\Omega^{n-1}k\rightarrow L\rightarrow\Omega^{-2}\Omega^{n}k\rightarrow0.
\end{equation}
As $\Omega^nk$ is $n$-torsion free and $\x$ is closed under $\Omega^{-n}$, it is easy to see that $\Omega^{-i}\Omega^{n}k\in\x$ for all $i\geq0$.
In particular, $\Omega^{-2}\Omega^{n}k\in\x$. It follows from the exact sequence (\ref{t6}.2) that $L\in\x$. As $\depth R=n-1$ and $\x\subseteq\Omega^n(\md(R))$, by Auslander-Buchsbaum formula, we conclude that $L$ is free. The exact sequence (\ref{t6}.2) induces the following exact sequence:
$$\cdots\rightarrow\Ext^1_R(L,R)\rightarrow\Ext^1_R(\Omega^{n-1}k,R)\rightarrow\Ext^2_R(\Omega^{-2}\Omega^{n}k,R)\rightarrow\cdots.$$
Hence $\Ext^n_R(k,R)=0$, and $R$ is Gorenstein of dimension $n-1$.
Thus $\x\subseteq\Omega^n(\md(R))\subseteq\g(R)$.
\end{proof}
An $R$-module $M$ is said to satisfy the property $\widetilde{S}_n$ if $\depth_{R_\fp}(M_\fp)\geq\min\{n, \depth R_\fp\}$ for all $\fp\in\Spec R$. Here is an example of a resolving subcategory which is closed under $\Omega^{-n}$ for some but not all $n>0$.
\begin{ex}\label{exa}
Let $R$ be a Gorenstein local ring of dimension $d>1$. For an integer $n$, $0<n<d$, set
$\x=\{M\in\md(R)\mid M \text{ satisfies } \widetilde{S}_n\}.$	
Then $\x$ is a resolving subcategory of $\md(R)$ and the following statements hold.
\begin{enumerate}[(i)]
		\item{$\g(R)\subsetneq\x$.}
		\item{$\Omega^{-m}\x\nsubseteq\x$ for $1\leq m\leq n$.}
		\item{$\Omega^{-i}\x\subseteq\x$ for all $i\geq d$.}
		\item{$\rad(\x)=\infty$.}
\end{enumerate}
\end{ex}
\begin{proof}
First note that by \cite[Theorem 42]{M1}, $\x=\Omega^n(\md(R))$.
It is easy to see that $\x$ is a resolving subcategory of $\md(R)$. Note that $\Omega^{n}k\in\x\setminus\g(R)$ and so (i)	is clear. For the proof of (ii), assume contrarily that
	$\Omega^{-m}\x\subseteq\x$ for some $m$, $1\leq m\leq n$.
Then by Corollary \ref{t6}, $\x\subseteq\g(R)$ which is a contradiction. Hence, $\Omega^{-m}\x\nsubseteq\x$ for $1\leq m\leq n$.
As for (iii), as $R$ is Gorenstein, $\Omega^i\Tr M\in\g(R)$ for all $M\in\md(R)$ and $i\geq d$. Therefore, by Proposition \ref{G} and part (i), $\Omega^{-i}M\cong\Tr\Omega^{i}\Tr M\in\g(R)\subseteq\x$.
Part (iv) follows from Corollary \ref{c4} and part (iii).
\end{proof}
The following corollary of Theorem \ref{t5}
should be compared with Corollary \ref{c3}.
\begin{cor}
Let $R$ be a Golod local ring of positive dimension and
let $\x$ be a non-trivial resolving subcategory of $\md(R)$ of finite radius. Then $\x$ is closed under cosyzygies if and only if $R$ is a hypersurface.
\end{cor}
\begin{proof}
Let $\x$ be closed under cosyzygies. As $\x$ is non-trivial,  by Theorem \ref{t5} and \cite[Theorem 4.4]{DT}, $\add(R)\subsetneq\x\subseteq\g(R)$.
It follows from \cite[Example 3.5]{AvM} that $R$ is hypersurface.

Now assume that $R$ is a hypersurface. As $\rad_R(\x)<\infty$,
by \cite[Theorem I]{DT}, $\x\subseteq\g(R)$. Hence, by \cite[Theorem 6.1]{E}, $\Omega^{-1}M\cong\Omega^{2}\Omega^{-1}M\cong\Omega M$ for all $M\in\x$. Thus, $\x$ is closed under cosyzygies.
\end{proof}
The following is an immediate consequence of Corollary \ref{t6} and \cite[Proposition 1.7]{DT}.
\begin{cor}
Let $R$ be a local ring and let $\x$ be a resolving subcategory of $\md(R)$. The following are equivalent:
\begin{enumerate}[(i)]
	\item  $\Omega^{-n}\x\subseteq\x\subseteq\Omega^n\x$ for some $n$, $1\leq n\leq\depth R+1$.
	\item  $\x$ is a thick subcategory of $\g(R).$
\end{enumerate}
\end{cor}
The following result is similarly proved to Claim I of Theorem \ref{t5}.
\begin{prop}\label{tt1}
	Let $(R,\fm,k)$ be a local ring of depth $t$ and let $\x$ be a resolving subcategory of $\md(R)$.
	If $\Omega^{t+1}\Omega^{-(t+1)}\x\subseteq\x$, then either $\Omega^tk\in\x$ or $\Ext^{t+2}_R(\x,R)=0$.	
\end{prop}
The following is an immediate consequence of
Proposition \ref{tt1} and Lemma \ref{t7}.
\begin{cor}
Let $(R,\fm,k)$	be a local ring of depth $t$ which is locally Gorenstein on $\XX^{t-1}(R)$. Let $\x$ be a resolving subcategory of $\md R$.
If $\x\subseteq\Omega^{t+1}(\md R)$, then either $\Omega^tk\in\x$ or $\Ext^{t+2}_R(\x,R)=0$.
\end{cor}
Let $\lambda$ be the linkage functor, that is, $\lambda=\Omega\Tr$.
Proposition \ref{tt1} gives rise to the following.
\begin{cor}
Let $(R,\fm,k)$ be a local ring and $\x$ a resolving subcategory of $\md R$.
The following statements hold:
\begin{enumerate}[(i)]
	\item{Assume $\depth R=0$. If $\lambda^2\x\subseteq\x$ (e.g. $\x\subseteq\Omega(\md R)$), then either $k\in\x$ or $\Ext^2_R(\x,R)=0$.}
	\item{Assume $\depth R=1$. If $\x^{**}\subseteq\x$ (e.g. $\x\subseteq\Omega^2(\md R)$ and $R$ is generically Gorenstein), then either $\fm\in\x$ or $\Ext^3_R(\x,R)=0$.}
\end{enumerate}
\end{cor}
\begin{cor}
	Let $R$ be a Golod local ring which is not a hypersurface.  Let $\x$ be a non-trivial resolving subcategory of $\md(R)$ such that $\Omega^{(t+1)}\Omega^{-(t+1)}\x\subseteq\x$, where $t=\depth R$. If $\rad(\x)<\infty$, then $R$ is Cohen--Macaulay and $\x\subseteq\cm(R)$.
\end{cor}
\begin{proof}
	By Proposition \ref{tt1}, either $\Omega^tk\in\x$ or $\Ext^{t+2}_R(\x,R)=0$. If $\Ext^{t+2}_R(\x,R)=0$, then either $R$ is Gorenstein or $\x=\add(R)$ by \cite[Proposition 1.4]{JS} and Lemma \ref{t}.
	As $R$ is not a hypersurface and $\x$ is non-trivial, we get a contradiction. Therefore, $\Omega^tk\in\x$.
	Now the assertion follows from Proposition \ref{t2} and Lemma \ref{l2}.
\end{proof}

\section*{Acknowledgments}
This work was done when the first author
visited Nagoya University in June and July 2017. He is grateful for the kind hospitality of the NU Department of Mathematics.

\bibliographystyle{amsplain}

\begin{thebibliography}{9}
\bibitem{A1}
~M. Auslander, \emph{Anneaux de Gorenstein, et torsion en alg\`{e}bre commutative}, in: S\'{e}minaire d'Alg\`{e}bre
Commutative dirig\'{e} par Pierre Samuel, vol. 1966/67, Secr\'{e}tariat math\'{e}matique, Paris, 1967.

\bibitem{AB}
~M. Auslander and ~M. Bridger, \emph{Stable module theory}, Mem. of
the AMS 94, Amer. Math. Soc., Providence 1969.

\bibitem{Av}
~L.~ L. Avramov, \emph{Homological asymptotics of modules over local rings}, Commutative algebra; Berkeley, 1987 (M. Hochster, C. Huneke, J. Sally, eds.), MSRI Publ. 15, Springer, New York 1989; pp. 33--62.

\bibitem{AvM}
~L.~ L. Avramov and ~A. Martsinkovsky, \emph{Absolute, relative, and
{T}ate cohomology of modules of finite {G}orenstein dimension}, Proc. London
Math. Soc. (3) 85 (2002), no.~2, 393--440.

\bibitem{BH}
~W. Bruns and ~J. Herzog, \emph{Cohen--Macaulay Rings}, Cambridge
Studies in Advanced Mathematics, 39. Cambridge University Press,
Cambridge, 1993.

\bibitem{DT1}
~H. Dao and ~R. Takahashi, \emph{Classification of resolving subcategories and grade consistent functions}, Int. Math. Res. Not. IMRN 2015(1), 119--149 (2015).

\bibitem{DT}
~H. Dao and ~R. Takahashi, \emph{The radius of a subcategory of modules}, Algebra Number Theory 8(1) (2014), 141--172.

\bibitem{E}
~D. Eisenbud, \emph{Homological algebra on a complete intersection, with an application to group
representations}, Trans. Amer. Math. Soc. 260 (1980), 35--64.

\bibitem{GP}
~V. N. Gasharov, ~I. V. Peeva, \emph{Boundedness versus periodicity over commutative local rings}, Trans. Amer. Math. Soc, 320 (1990), 569--580.

\bibitem{GOTWY}
~S. Goto, ~K. Ozeki, ~R. Takahashi, ~K. Watanabe, ~K. Yoshida,
\emph{Ulrich ideals and modules},
Math. Proc. Cambridge Philos. Soc. 156 (2014), no. 1, 137--166.

\bibitem{GW}
~S. Goto and ~K.-i. Watanabe, \emph{On graded rings, I}, J. Math. Soc. Japan, 309(1978), 179--213.

\bibitem{JS}
~D. Jorgensen, ~L.~ M. \c{S}ega, \emph{Nonvanishing cohomology and classes
of Gorenstein rings}, Advances in Mathematics 188 (2004) 470--490.

\bibitem{M1}
~V. Ma\c{s}iek, \emph{Gorenstein dimension and torsion of modules
over commutative Noetherian rings}, Comm. Algebra (2000), 5783--5812.

\bibitem{NT}
~S. Nasseh and ~R. Takahashi. \emph{Local rings with quasi-decomposable maximal ideal}. Preprint.

\bibitem{T}
~R. Takahashi, \emph{Classifying thick subcategories of the stable category of Cohen--Macaulay modules.} Advances in Mathematics. 225, no. 4 (2010) 2076--116.

\bibitem{T1}
~R. Takahashi, \emph{Modules in resolving subcategories which are free on the punctured spectrum}, Pacific J. Math. 241 (2009), no. 2, 347--367.

\bibitem{Y}
~Y. Yoshino, \emph{A functorial approach to modules of G-dimension zero.} Ill. J. Math. 49(2), 345--367 (2005).

\end{thebibliography}

\end{document}